\documentclass[10pt,leqno]{amsart}
   \usepackage[centertags]{amsmath}
   \usepackage{amsfonts}
   \usepackage{amsmath}
   \usepackage{amssymb}
   \usepackage{amsthm}
   \usepackage{newlfont}
   \usepackage[latin1]{inputenc}
\usepackage{multirow, bigdelim}

\usepackage{graphicx}

\usepackage{bm}

\newcommand{\dosfilas}[2]{
  \ldelim[{2}{2mm}& #1 &\rdelim]{2}{2mm} \\
  & #2 & &  & &
}

\allowdisplaybreaks[3]

\theoremstyle{plain}
\newtheorem{theorem}{Theorem}[section]
\newtheorem{lemma}[theorem]{Lemma}

\newtheorem{corollary}[theorem]{Corollary}

\theoremstyle{definition}
\newtheorem{definition}[theorem]{Definition}

\theoremstyle{remark}
\newtheorem{remark}[theorem]{Remark}
\newtheorem*{remark*}{Remark}

\numberwithin{equation}{section}

\newcommand\D{{\mathcal D}}
\newcommand\A{{\mathcal A}}

\newcommand\I{{\mathcal I}}
\newcommand\U{{\mathcal{U}}}

\newcommand\cY{{\mathcal Y}}

\newcommand\RR{{\mathbb R}}
\newcommand\ZZ{{\mathbb Z}}
\newcommand\NN{{\mathbb N}}
\newcommand\PP{{\mathbb P}}
\newcommand\II{{\mathbb I}}
\newcommand\UU{{\mathbb U}}

\newcommand\no{{\rho}}

\newcommand\awr{\operatorname{\mbox{$\alpha$}-wr}}
\newcommand\bwr{\operatorname{\mbox{$\beta$}-wr}}
\newcommand\gwr{\operatorname{\mbox{$\gamma$}-wr}}

\newcommand\rank{\operatorname{rank}}

\title[]{Differential equations for discrete \\ Jacobi-Sobolev orthogonal polynomials}

\author{Antonio J. Dur\'an and Manuel D. de la Iglesia}

\address{Antonio J. Dur\'an\\
Departamento de An\'{a}lisis Matem\'{a}tico \\
Universidad de Sevilla \\
Apdo (P. O. BOX) 1160\\
41080 Sevilla. Spain.}
\email{duran@us.es}
\address{Manuel D. de la Iglesia\\
Instituto de Matem\'aticas \\
Universidad Nacional Aut\'onoma de M\'exico \\
Circuito Exterior, C.U.\\
04510, Mexico D.F. Mexico.}
\email{mdi29@im.unam.mx}
\thanks{Partially supported by MTM2012-36732-C03-03 (Ministerio de Econom\'ia y Competitividad),
FQM-262, FQM-4643, FQM-7276 (Junta de Andaluc\'ia) , Feder Funds (European Union) and PAPIIT-DGAPA-UNAM grant IA100515 (México).}

\date{\today}



\begin{document}
  \maketitle

   \begin{abstract}
The aim of this paper is to study differential properties of orthogonal polynomials with respect to a discrete Jacobi-Sobolev bilinear form
with mass point at $-1$ and/or $+1$. In particular, we construct the orthogonal polynomials using certain Casorati determinants. Using this construction, we prove that when the Jacobi parameters $\alpha$ and $\beta$ are nonnegative integers the Jacobi-Sobolev orthogonal polynomials are eigenfunctions of a differential operator of finite order (which will be explicitly constructed). Moreover, the order of this differential operator is explicitly computed in terms of the matrices which define the discrete Jacobi-Sobolev bilinear form.
\end{abstract}

\section{Introduction and main results}
Classical polynomials are orthogonal polynomials (with respect to a positive measure) which are in addition eigenfunctions of a second-order differential operator. As a consequence of S. Bochner classification theorem of 1929 (see \cite{B}), it follows that there are only three families of classical
orthogonal polynomials: Hermite, Laguerre and Jacobi (and Bessel
polynomials if signed measures are considered). Although such result actually goes back to E.J. Routh in 1884 (see \cite{Rou}).

H.L. Krall raised in 1939 (see \cite{Kr1,Kr2}) the problem of finding
orthogonal polynomials which are also common eigenfunctions of a
higher-order differential operator with polynomial coefficients.
He obtained a complete classification for the case of a
differential operator of order four (see \cite{Kr2}). Besides the
classical families of Hermite, Laguerre and Jacobi, he  found three other
families of orthogonal polynomials which are also eigenfunctions of a fourth-order
differential operator. Two of them are orthogonal
with respect to positive measures which consist of particular instances of Jacobi weights together with one or two Dirac deltas at the endpoints of the interval of
orthogonality. Indeed, consider the Koornwinder measures (see \cite{Ko})
\begin{equation}\label{jacobid}
(1-x)^\alpha (1+x)^\beta +M\delta_{-1}+N\delta_1,\quad \alpha, \beta >-1.
\end{equation}
Then, the examples found by Krall correspond with the cases $\alpha=\beta=0$ and $M=N$, and $\beta=N=0$ in (\ref{jacobid}), respectively. Krall also discovered a new family satisfying sixth-order differential equations, which corresponds with the case $\alpha=\beta=0$ in
(\ref{jacobid}). But he never published this example which was rediscovered by
L.L. Littlejohn forty years later (see \cite{L1}).

R. Koekoek proved in 1994 that the Koornwinder polynomials orthogonal with respect to the weight (\ref{jacobid}) for $\alpha =\beta\in \NN $ and $M=N$ are also eigenfunctions of a differential operator of order $2\alpha+4$ (see \cite{koe}). F.A. Grünbaum and L. Haine (et al.) proved that polynomials satisfying fourth or higher-order differential equations can be generated by applying Darboux transformations to certain instances of the classical polynomials (see \cite{GrH1,GrHH,GrY}). A. Zhedanov proposed a technique to construct Krall's polynomials and found a differential equation of order $2\alpha +4$ for the orthogonal polynomials with respect to (\ref{jacobid}) when $\alpha$ is a nonnegative integer and $M=0$ (see \cite{Zh}).
R. and J. Koekoek proved in 2000 the general case. More precisely, they found a differential operator for the orthogonal polynomials with respect to the weight (\ref{jacobid}) for which they are eigenfunctions; this operator has infinite order except for the following cases, where the order is finite and equals (see \cite{koekoe2}):
$$
\begin{cases} 2\beta +4, & \mbox{if $M>0,N=0$ and $\beta\in \NN$},\\
2\alpha +4, & \mbox{if $M=0,N>0$ and $\alpha\in \NN$},\\
2\alpha+2\beta +6, & \mbox{if $M>0,N>0$ and $\alpha,\beta\in \NN$}.
\end{cases}
$$
Using a different approach, P. Iliev (see \cite{Plamen1}) has improved this result by studying the algebra of differential operators associated with Krall-Jacobi orthogonal polynomials.

In 2003, discrete Jacobi-Sobolev orthogonal polynomials which are also common eigenfunctions of a higher-order differential operator entered into the picture. H. Bavinck (see \cite{Ba}) proved that orthogonal polynomials with respect to the discrete Jacobi-Sobolev inner product
$$
\langle p,q\rangle=\int_{-1}^1
p(x)q(x)(1-x)^{\alpha}(1+x)^\beta dx+Mp^{(l_1)}(-1)q^{(l_1)}(-1)+Np^{(l_2)}(1)q^{(l_2)}(1),
$$
are eigenfunctions of a differential operator of infinite order, except for the following cases, where the order is finite and equals:
$$
\begin{cases} 2\beta +4l_1+4, & \mbox{if $M>0,N=0$ and $\beta\in \NN$},\\
2\alpha +4l_2+4, & \mbox{if $M=0,N>0$ and $\alpha\in \NN$},\\
2\alpha+2\beta +4l_1+4l_2+6, & \mbox{if $M>0,N>0$ and $\alpha,\beta\in \NN$}.
\end{cases}
$$
For other related papers see \cite{JKLL1,JKLL2,Mar}.

\medskip

For $\alpha,\beta,\alpha+\beta \not =-1,-2,\ldots,$ we use the
following definition of the Jacobi polynomials:
\begin{align*}
J_n^{\alpha,\beta}(x)&=\frac{(-1)^n(\alpha+\beta+1)_n}{2^n(\beta+1)_n}\sum_{j=0}^n\binom{n+\alpha}{j}\binom{n+\beta}{n-j}(x-1)^{n-j}(x+1)^j\\
&=\frac{(-1)^n(\alpha+\beta+1)_n(\alpha+1)_n}{n!(\beta+1)_n}{}_2F_1\left(-n,n+\alpha+\beta+1; \alpha+1;\frac{1-x}{2}\right),
\end{align*}
where $(a)_n=a(a+1)\cdots(a+n-1)$ denotes the Pochhammer symbol. We use a different normalization of the standard definition of the Jacobi polynomials $P_n^{\alpha,\beta}$ and the equivalence is given by $J_n^{\alpha,\beta}(x)=(-1)^n\frac{(\alpha+\beta+1)_n}{(\beta+1)_n}P_n^{\alpha,\beta}(x)$ (these and the next formulas can be found in \cite{Bate} pp. 168--173).

We denote by $\mu_{\alpha,\beta}(x)$ the orthogonalizing weight for the Jacobi polynomials normalized so that 
$
\int \mu_{\alpha,\beta}(x)dx=2^{\alpha+\beta+1}\frac{\Gamma(\alpha+1)\Gamma(\beta+1)}{\Gamma(\alpha+\beta +1)}.
$
Only when $\alpha,\beta >-1$, $\mu_{\alpha,\beta }(x)$, $-1<x<1$, is positive, and then
\begin{equation}\label{pJac}
\mu_{\alpha,\beta}(x) =(1-x)^\alpha (1+x)^\beta,\quad -1<x<1.
\end{equation}

The Jacobi polynomials $(J_n^{\alpha,\beta})_n$ are eigenfunctions of the following second-order differential operator
\begin{equation}\label{dopJac}
D_{\alpha,\beta}=(x^2-1)\left(\frac{d}{dx}\right)^2+((\alpha+\beta+2)x-\beta+\alpha)\frac{d}{dx}.
\end{equation}
That is
\begin{equation}\label{autov}
D_{\alpha,\beta}(J_n^{\alpha,\beta})=\theta_nJ_n^{\alpha,\beta},\quad\theta_n=n(n+\alpha+\beta+1), \quad n\geq0.
\end{equation}
For $m_1,m_2\geq0$ with $m=m_1+m_2\geq1$, let $\bm M=(M_{i,j})_{i,j=0}^{m_1-1}$ and $\bm N=(N_{i,j})_{i,j=0}^{m_2-1}$ be $m_1\times m_1$ and $m_2\times m_2$ matrices, respectively. In particular, if $m_1=0$ or $m_2=0$ we take $\bm M=0$ or $\bm N=0$, respectively.
We consider the discrete Jacobi-Sobolev bilinear form defined by
\begin{equation}\label{idslip}
\langle p,q\rangle=\int _ {I} p(x)q(x)\mu_{\alpha-m_2,\beta-m_1} dx+\mathbb{T}^{m_1}_{-1 }(p)\bm M\mathbb{T}^{m_1}_{-1}(q)^T+\mathbb{T}^{m_2}_{1}(p)\bm N\mathbb{T}^{m_2}_{1}(q)^T,
\end{equation}
where for a nonnegative integer $k$, a real number $\lambda$ and a polynomial $p$, we define
$\mathbb{T}^{k}_{\lambda }(p)=\left(p(\lambda),p'(\lambda),\ldots,p^{(k-1)}(\lambda)\right)$. For $\alpha>m_2-1$ and $\beta>m_1-1$, the measure $\mu_{\alpha-m_2,\beta-m_1}$ in \eqref{idslip} is then $(1-x)^{\alpha-m_2}(1+x)^{\beta-m_1}$ and  $I= (-1,1)$.

The purpose of this paper is to prove in a constructive way that
if $\alpha$ and $\beta$  are nonnegative integers with $\alpha \ge m_2$ and $\beta\ge m_1$, then the orthogonal polynomials with respect to (\ref{idslip}) are eigenfunctions of a finite order differential operator with polynomial coefficients. For discrete Laguerre-Sobolev orthogonal polynomials see \cite{ddI1}.

To display our results in full detail we need some notation.
For $a,b,c,d,x\in \RR$, we define
\begin{equation}\label{gammas}
\Gamma^{a,b}_{c,d}(x)=\frac{\Gamma(x+a+1)\Gamma(x+b+1)}{\Gamma(x+c+1)\Gamma(x+d+1)}.
\end{equation}
We next introduce the functions $z_l$, $1\le l\le m,$ in the following way. For $l=1,\ldots, m_1$, $z_l$ is defined by
\begin{align}\label{ri1}
z_l(x)&=\frac{2^{\alpha+\beta-m_1+l}\Gamma(\beta-m_1+l)}{(m_1-l)!} \Gamma^{m_1-l,\alpha+\beta}_{0,\alpha+\beta-m_1+l}(x)\\
\nonumber &\qquad\quad+2^{m_2}\sum_{i=0}^{m_1-1}\left(\sum_{j=l}^{(l+m_2)\wedge m_1}\frac{(j-1)!\binom{m_2}{j-l}M_{i,j-1}}{(-2)^{i+j-l}}\right)
\frac{\Gamma^{\beta,\alpha+\beta+i}_{-i,\alpha}(x)}{\Gamma(\beta+i+1)},
\end{align}
where $\wedge$ denotes the minimum between two numbers. For $l=m_1+1,\ldots, m$, $z_l$ are defined by
\begin{align}\label{ri2}
z_l(x)&=\frac{2^{\alpha+\beta-m+l}\Gamma(\alpha-m+l)}{(m-l)!}
\Gamma^{m-l,\alpha+\beta}_{0,\alpha+\beta-m+l}(x)\\
\nonumber &\qquad\quad +\sum_{i=0}^{m_2-1}\left(\sum_{j=l-m_1}^{l\wedge m_2}\frac{(j-1)!\binom{m_1}{l-j}N_{i,j-1}}{(-1)^{l-m_1-1}2^{i+j-l}}\right)
\frac{\Gamma^{\alpha,\alpha+\beta+i}_{-i,\beta}(x)}{\Gamma(\alpha+i+1)}.
\end{align}
Finally, we also need the polynomials
\begin{align}\label{def1pi}
p(x)&=\prod_{i=1}^{m_1-1}(-1)^{m_1-i}(x+\alpha-m+1)_{m_1-i}(x+\beta-m_1+i)_{m_1-i},\\
\label{def1qi}q(x)&=(-1)^{\binom{m}{2}}\prod_{h=1}^{m-1}\left(\prod_{i=1}^{h}(2(x-m)+\alpha+\beta+i+h)\right).
\end{align}
Using a general result for discrete Sobolev bilinear forms
(see Lemma \ref{LemJacSob} in Section \ref{sec2}), we first
characterize the existence of (left) orthogonal polynomials with
respect to the Jacobi-Sobolev bilinear form above using the
(quasi) Casorati determinant defined by the sequences $z_l$,
$l=1,\ldots , m$. Moreover, we find a closed expression for these
orthogonal polynomials in terms of the Jacobi polynomials
$(J_n^{\alpha,\beta})_n$ and the sequences $z_l$, $l=1,\ldots, m$.

\begin{theorem}\label{mainth1}
For $m_1,m_2\geq0$ with $m=m_1+m_2\geq1$, let $\bm M=(M_{i,j})_{i,j=0}^{m_1-1}$ and $\bm N=(N_{i,j})_{i,j=0}^{m_2-1}$ be $m_1\times m_1$ and $m_2\times m_2$ matrices, respectively. For $\alpha\neq m_2-1, m_2-2, \ldots,$ and $\beta\neq m_1-1, m_1-2, \ldots,$ consider the discrete Jacobi-Sobolev bilinear form defined by (\ref{idslip}). If we write
\begin{equation}\label{defno}
\no ^h_{x,j}=\begin{cases} (-1)^{m-j}\Gamma^{\alpha-j,\beta-1}_{\alpha-m,\beta-j}(x),&\mbox{for $h=1,\ldots, m_1$,}\\
1,&\mbox{for $h=m_1+1,\ldots ,m$,}
\end{cases}
\end{equation}
then the following conditions are equivalent
\begin{enumerate}
\item The discrete Jacobi-Sobolev bilinear form (\ref{idslip}) has a sequence $(q_n)_n$ of (left) orthogonal polynomials.
\item The $m\times m$ (quasi) Casorati determinant
\begin{equation}\label{casdet}
\Lambda (n)=\frac{\begin{vmatrix}
\no^1_{n,1} z_1(n-1) &\no^1_{n,2} z_1(n-2) & \cdots &  \no^1_{n,m}z_1(n-m) \\
               \vdots & \vdots & \ddots & \vdots \\
\no^m_{n,1} z_m(n-1) & \no^m_{n,2} z_m(n-2) & \cdots & \no^m_{n,m}z_m(n-m)
             \end{vmatrix}}{p(n)q(n)},
\end{equation}
where $z_l, l=1,\ldots, m,$ $p$ and $q$ are defined by (\ref{ri1}), (\ref{ri2}), (\ref{def1pi}) and (\ref{def1qi}) respectively, does not vanish for $n\geq0$.
\end{enumerate}
Moreover, if one of these properties holds, the polynomials defined by
\begin{equation}\label{iquss}
q_n(x)=\frac{\begin{vmatrix}
               J_{n}^{\alpha,\beta}(x) & -J_{n-1}^{\alpha,\beta}(x) & \cdots & (-1)^mJ_{n-m}^{\alpha,\beta}(x) \\
\no^1_{n,0}z_1(n) &\no^1_{n,1} z_1(n-1) & \cdots &  \no^1_{n,m}z_1(n-m) \\
               \vdots & \vdots & \ddots & \vdots \\
\no^m_{n,0}z_m(n) & \no^m_{n,1} z_m(n-1) & \cdots & \no^m_{n,m}z_m(n-m)
             \end{vmatrix}}{p(n)q(n)},
\end{equation}
are orthogonal with respect to (\ref{idslip}) (as usual for $n<0$ we take $J_n^{\alpha,\beta} =0$).
\end{theorem}

As in \cite{ddI1}, we find the differential properties of the orthogonal polynomials
$(q_n)_n$ (\ref{iquss}) by using the concept of \emph{$\D$-operators}.
This is an abstract concept introduced by one of us in
\cite{du1} which has shown to be very useful to generate
orthogonal polynomials which are also eigenfunctions of
differential, difference or $q$-difference operators (see
\cite{AD}, \cite{du1}--\cite{ddI2}). The basic facts about $\D$-operators will be
recalled in Section \ref{sec3}. Using the general theory of $\D$-operators
and the expression (\ref{iquss}) for the orthogonal polynomials
$(q_n)_n$, we prove in Section \ref{sec4} the following theorem.

\begin{theorem}\label{mainth2}
Assume that any of the two equivalent properties (1) and (2) in Theorem \ref{mainth1} holds. If $\bm M ,\bm N \not =0$, we assume, in addition, that $\alpha$ and $\beta$ are nonnegative integers with $\alpha \ge m_2$ and $\beta \ge m_1$. If, instead, $\bm M=0$,  we assume that only $\alpha$ is a positive integer with $\alpha \ge m_2$, and
if $\bm N=0$, we assume that only $\beta$ is a positive integer with $\beta \ge m_1$. Then there exists a finite order differential operator (which we construct explicitly) for which the orthogonal polynomials $(q_n)_n$ (\ref{iquss}) are eigenfunctions.
\end{theorem}

An important issue will be the explicit calculation of the order of this differential operator in terms of the matrices $\bm M$ and $\bm N$ which define the discrete Jacobi-Sobolev bilinear form (\ref{idslip}). As in \cite{ddI1}, the key concept to calculate that order will be the \emph{$\gamma$-weighted rank} associated to a real number $\gamma$ and a matrix $M$ defined as follows:

\begin{definition}\label{wrM}
Let $\gamma$ and $M$ be a real number and a $m\times m$ matrix,
respectively. Write $c_1,\ldots , c_m$, for the columns of $M$ and
define the numbers $\eta_j$, $j=1,\ldots , m,$ by
$$
\eta_1=\begin{cases} \gamma +m-1,& \mbox{if \;\;$c_m\not =0$,}\\ 0,&\mbox{if \;\;$c_m=0$;}\end{cases}
$$
and for $j=2,\ldots , m,$
\begin{equation*}\label{nj}
\eta_j=\begin{cases} \gamma +m-j,& \mbox{if \;\;$c_{m-j+1}\not \in <c_{m-j+2},\ldots , c_m>$,}\\ 0,&\mbox{if \;\;$c_{m-j+1} \in <c_{m-j+2},\ldots , c_m>$.}\end{cases}
\end{equation*}
Denote by $\tilde M$ the matrix whose columns are $c_i$, $i\in \{ j:\eta_{m-j+1}\not=0\}$ (i.e., the columns of $\tilde M$ are (from right to left) those columns $c_i$ of $M$ such that $c_i\not \in <c_{i+1},\ldots , c_m>$). Write $f_1,\ldots , f_m,$ for  the rows of $\tilde M$. We define the numbers $\tau_j$, $j=1,\ldots , m-1,$ by
\begin{equation*}\label{mj}
\tau_j=\begin{cases} m-j,& \mbox{if \;\;$f_{j} \in <f_{j+1},\ldots , f_m>$,}\\ 0,&\mbox{if \;\;$f_{j}\not \in <f_{j+1},\ldots , f_m>$.}\end{cases}
\end{equation*}
The \emph{$\gamma$-weighted rank} of the matrix $M$, $\gwr (M)$ in short, is then defined by
$$
\gwr(M)=\sum_{j=1}^m \eta_j+\sum_{j=1}^{m-1}\tau_j-\binom{m}{2}.
$$
\end{definition}
We then have the following
\begin{corollary}\label{cori1}
With the assumptions of Theorem \ref{mainth2}, the minimal order of the differential operators having the orthogonal polynomials $(q_n)_n$ as eigenfunctions is at most $2(\bwr (\bm M)+\awr (\bm N)+1)$.
\end{corollary}

Actually, for any nonnegative integer $l\ge 0$, we will construct
a differential operator of order $2(l+\bwr (\bm M)+\awr (\bm
N)+1)$ for which the orthogonal polynomials $(q_n)_n$ are
eigenfunctions (see Theorem \ref{mainth3}). We have computational
evidences showing that, except for special values of the
parameters $\alpha$ and $\beta$ and the matrices $\bm M$ and $\bm
N$, the minimum order of a differential operator having the
orthogonal polynomials $(q_n)_n$ as eigenfunctions seems to be
$2(\bwr (\bm M)+\awr (\bm N)+1)$. However this is not true in
general. For instance, when $\alpha=\beta$, $m_1=m_2=1$ and $\bm
M=\bm N$, we have that $2(\bwr (\bm M)+\awr (\bm N)+1)=4\alpha
+2\; (\alpha\geq1).$ But Koekoek found a differential operator of
order $2\alpha +2 \; (\alpha\geq1)$ for which these Gegenbauer
type orthogonal polynomials are eigenfunctions (see \cite{koe}).
However, we will show that our method can be adapted to this and
other special cases and provides differential operators of order
lower than $2(\bwr (\bm M)+\awr (\bm N)+1)$.


We finish pointing out that, as explained above, the approach of this paper is the same as in \cite{ddI1} for discrete Laguerre-Sobolev orthogonal polynomials. Since we work here with two matrices $\bm M$ and $\bm N$ (instead of only one matrix as in \cite{ddI1}), and the sequence of eigenvalues for the Jacobi polynomials (\ref{autov}) is a quadratic polynomial in $n$, the computations are technically more involved. Therefore, we will omit some proofs which are too similar to the corresponding ones in \cite{ddI1}.

\section{Preliminaries and proof of Theorem \ref{mainth1}}\label{sec2}
We say that a sequence of polynomials $(q_n)_n$, with $\deg(q_n)=n,n\ge 0$, is (left) orthogonal with respect to a bilinear form $B$ (not necessarily symmetric) defined in the linear space of real polynomials if $B(q_n,q)=0$ for all
polynomials $q$ with $\deg(q)<n$ and $B(q_n,q_n)\not
=0$. It is clear from the definition that (left) orthogonal polynomials with respect to a bilinear form, if they exist, are unique up to multiplication by nonzero constants. Given a measure $\nu$ (positive or not), with finite moments of any order, we consider the bilinear form $B_\nu (p,q)=\int pqd\nu$. We then say that a sequence of polynomials $(q_n)_n$, with $\deg(q_n)=n, n\ge 0$, is orthogonal with respect to the measure $\nu$ if it is orthogonal with respect to the bilinear form $B_\nu$.

We will use the following lemma to construct (left) orthogonal
polynomials with respect to a discrete Sobolev bilinear form with two nodes. This result is an extension of the Lemma 2.1 of  \cite{ddI1}.

\begin{lemma}\label{LemJacSob}
For $m_1,m_2\geq0$ with $m=m_1+m_2\geq1$, let $\bm M=(M_{i,j})_{i,j=0}^{m_1-1}$ and $\bm N=(N_{i,j})_{i,j=0}^{m_2-1}$ be $m_1\times m_1$ and $m_2\times m_2$ matrices, respectively. For a given measure $\nu$ and for a couple of real numbers $\lambda$ and $\mu$ ($\lambda\neq\mu$) consider the discrete Sobolev bilinear form defined by
\begin{equation}\label{DSip}
\langle p,q\rangle=\int p(x)q(x)d\nu(x)+\mathbb{T}^{m_1}_{\lambda}(p)\bm M\mathbb{T}^{m_1}_{\lambda}(q)^T+\mathbb{T}^{m_2}_{\mu }(p)\bm N\mathbb{T}^{m_2}_{\mu }(q)^T,
\end{equation}
where for a nonnegative integer $k$, a real number $\lambda$ and a polynomial $p$, we define
$\mathbb{T}^k_\lambda(p)=\left(p(\lambda),p'(\lambda),\ldots,p^{(k-1)}(\lambda)\right)$.
Assume that the measure $(x-\lambda)^{m_1}(\mu-x)^{m_2}\nu$ has a sequence
$(p_n)_n$ of orthogonal polynomials, and write
\begin{equation*}
w_{n,i}^1=\int(x-\lambda)^i(\mu-x)^{m_2}p_{n}(x)d\nu,\quad
w_{n,j}^2=\int(x-\lambda)^{m_1}(\mu-x)^j p_{n}(x)d\nu.
\end{equation*}
For $l=1,\ldots,m_1,$
define the sequences $(R_l(n))_n$ by
\begin{equation}\label{RdSip}
R_l(n)=w_{n,l-1}^1+\sum_{i=0}^{m_1-1}\left(\sum_{j=l}^{(l+m_2)\wedge m_1} \frac{(j-1)!\binom{m_2}{j-l}M_{i,j-1}}{(-1)^{m_2}(\lambda-\mu)^{j-l-m_2}}\right)p_n^{(i)}(\lambda),
\end{equation}
where $\wedge$ denotes the minimum between two numbers. For $l=m_1+1,\ldots,m,$ define the sequences $(R_l(n))_n$ by
\begin{equation}\label{RdSip2}
R_l(n)=w_{n,l-m_1-1}^2+\sum_{i=0}^{m_2-1}\left(\sum_{j=l-m_1}^{l\wedge m_2} \frac{(j-1)!\binom{m_1}{l-j}N_{i,j-1}}{(-1)^{l-m_1-1}(\mu-\lambda)^{j-l}}\right)p_n^{(i)}(\mu).
\end{equation}
Then the following conditions are equivalent:
\begin{enumerate}
\item For $n\ge m$, the discrete Sobolev bilinear form (\ref{DSip}) has an (left) orthogonal polynomial $q_n$ with $\deg (q_n)=n$ and nonzero norm.
\item The $m\times m$ Casorati determinant $\Lambda (n)=\det (R_i(n-j))_{i,j=1}^m$ does not vanish for $n\geq m$.
\end{enumerate}
Moreover, if one of these properties holds, the polynomial defined
by
\begin{equation*}\label{quss}
q_n(x)=\begin{vmatrix}
               p_n(x) & p_{n-1}(x) & \cdots & p_{n-m}(x) \\
R_1(n) &R_1(n-1) & \cdots & R_1(n-m) \\
               \vdots & \vdots & \ddots & \vdots \\
               R_m(n) &  R_m(n-1) & \cdots & R_m(n-m)
             \end{vmatrix},
\end{equation*}
has degree $n$, $n\ge m$, and the sequence $(q_n)_n$ is (left)
orthogonal with respect to (\ref{DSip}).
\end{lemma}

\begin{proof}
We proceed in the same lines as the proof of Lema 2.1 in \cite{ddI1}, but changing the powers $(x-\lambda)^l$, $l=1,\ldots ,m$, to
$(x-\lambda)^{i-1}(\mu-x)^{m_2}$, $i=1,\ldots ,m_1$ and $(x-\lambda)^{m_1}(\mu-x)^{j-1}$, $j=1,\ldots ,m_2$.
We only have to explain how to find the identities (\ref{RdSip}) and (\ref{RdSip2}) for the sequences $R_l(n)$, $l=1,\ldots,m.$

For $l=1,\ldots,m_1,$ we use $q(x)=(x-\lambda)^{l-1}(\mu-x)^{m_2}.$ Then, it is straightforward to see that every component $i$ of the vector $\mathbb{T}^{m_1}_{\lambda}(q)$ is given by
$$
[\mathbb{T}^{m_1}_{\lambda}(q)]_i=\begin{cases}
\frac{(i-1)!\binom{m_2}{i-l}}{(-1)^{m_2}(\lambda-\mu)^{i-l-m_2}},&\mbox{if}\quad i=l,\ldots,(l+m_2)\wedge m_1,\\
0&\mbox{if}\quad i=1,\ldots,l-1,\quad\mbox{or}\quad i=l+m_2+1,\ldots,m_1,
\end{cases}
$$
while $\mathbb{T}^{m_2}_{\mu }=(0,\ldots,0)$. Hence, if we write
$$
q_n(x)=\sum_{j=0}^m\beta_{n,j}p_{n-j}(x),
$$
where $\beta_{n,0}=1$, we get \eqref{RdSip} by evaluating the inner product $\langle q_n, (x-\lambda)^{l-1}(\mu-x)^{m_2}\rangle=0,$ for $l=1,\ldots,m_1.$

Similar for $l=m_1+1,\ldots,m,$ where now we use $q(x)=(x-\lambda)^{m_1}(\mu-x)^{l-m_1-1}.$ Then
$$
[\mathbb{T}^{m_2}_{\mu }(q)]_i=\begin{cases}
\frac{(i-1)!\binom{m_1}{l-i}}{(-1)^{l-m_1-1}(\mu-\lambda)^{i-l}},&\mbox{if}\quad i=l-m_1,\ldots,l\wedge m_2,\\
0&\mbox{if}\quad i=1,\ldots,l-m_1-1,\quad\mbox{or}\quad i=l+1,\ldots,m_2,
\end{cases}
$$
while $\mathbb{T}^{m_1}_{\lambda}(q)=(0,\ldots,0)$. Evaluating $\langle q_n, (x-\lambda)^{m_1}(\mu-x)^{l-m_1-1}\rangle=0,$ for $l=m_1+1,\ldots,m,$ we get \eqref{RdSip2}.

Finally, observe that the set of polynomials
\begin{equation}\label{base}
b_l(x)=\begin{cases}
(x-\lambda)^{l-1}(\mu-x)^{m_2},& l=1,\ldots,m_1,\\
(x-\lambda)^{m_1}(\mu-x)^{l-m_1-1}, &l=m_1+1,\ldots,m,
\end{cases}
\end{equation}
is linearly independent and has dimension exactly $m$ only when $\lambda\neq\mu$.
\end{proof}

We also need the following combinatorial identities (which can be proved using standard combinatorial techniques).

\begin{lemma}\label{comb} Let $\alpha,\beta$ be non-integers real numbers, such that $\alpha+\beta$ is not an integer. Let $m_1,m_2$ be nonnegative integers
with $m_1+m_2\geq1$ and write $m=m_1+m_2$.
\begin{enumerate}
\item For $k=1,\ldots,m_1-h-1$, $h=0,\ldots, m_1-2$,
\begin{equation}\label{comb1}
\sum_{l=0}^{m_1-1}\frac{(-1)^l\binom{h}{m_1-l}\binom{l-k}{l}}{2^l(\beta-l)\binom{\alpha+\beta-k-l}{\alpha-k}}=0.
\end{equation}
\item For $k=1,\ldots,m-1$,
\begin{equation}\label{comb2}
(-1)^k\sum_{l=0}^{m_1-1}\frac{\binom{m-l-2}{m_2-1}\binom{l-k}{l}}{(\beta-l)\binom{\alpha+\beta-k-l}{\alpha-k}}
+\sum_{l=0}^{m_2-1}\frac{\binom{m-l-2}{m_1-1}\binom{l-k}{l}}{(\alpha-l)\binom{\alpha+\beta-k-l}{\beta-k}}=0.
\end{equation}
\end{enumerate}
\end{lemma}

We are now ready to prove Theorem \ref{mainth1} in the Introduction.
\begin{proof}[Proof of Theorem \ref{mainth1}]
We proceed in two steps.

\medskip
\noindent
\textit{First step. Assume $n\ge m$.}
Notice that for $n\ge m$, $p(n)q(n)\not=0$, where $p$ and $q$ are the polynomials given by (\ref{def1pi}) and (\ref{def1qi}).
Actually, we can remove the normalization $1/(p(n)q(n))$ from the definition of the polynomials $q_n$ (as we will see below, this normalization is going to be useful only for some instances of $\alpha$ and $\beta$ when $n=0,\ldots, m-1$).

For $n\ge m$, the theorem is a direct consequence of Lemma \ref{LemJacSob} for $\lambda=-1$, $\mu=1$, the Jacobi measure $\nu=\mu_{\alpha-m_2,\beta-m_1}$ (\ref{pJac})
and the Jacobi polynomials $p_n=J_n^{\alpha,\beta}$.

Indeed, we use the following expansions (see Theorem 3.21, p.76 of \cite{STW}, after some computations using Pochhammer symbol properties):
\begin{align*}
J_n^{\alpha,\beta}(x)&=\sum_{k=0}^n \left[\frac{2k+\alpha+\gamma+1}{k+\alpha+\gamma+1}\frac{\binom{n+\alpha+\beta}{\alpha}\binom{\alpha+\gamma}{\gamma}\binom{n+\alpha}{n-k}\binom{n+\alpha+\beta+k}{k}\binom{n-k+\beta-\gamma-1}{n-k}}{\binom{\alpha+\beta}{\beta}\binom{k+\alpha+\gamma}{\alpha}\binom{n}{k}\binom{n+\alpha+\gamma+k+1}{n}}\right]J_k^{\alpha,\gamma}(x),\\
J_n^{\alpha,\beta}(x)&=\sum_{k=0}^n\left[\frac{2k+\beta+\gamma+1}{k+\beta+\gamma+1}\frac{\binom{n+\alpha+\beta}{\alpha}\binom{\beta+\gamma}{\gamma}\binom{n+\beta}{n-k}\binom{n+\alpha+\beta+k}{k}\binom{n-k+\alpha-\gamma-1}{n-k}}{(-1)^{n-k}\binom{\alpha+\beta}{\beta}\binom{k+\beta+\gamma}{\gamma}\binom{n}{k}\binom{n+\beta+\gamma+k+1}{n}}\right] J_k^{\gamma,\beta}(x).
\end{align*}
Therefore we get for $l=1,\ldots, m_1$,
\begin{align*}
w_{n,l-1}^1&=\int (x+1)^{l-1}J_n^{\alpha,\beta}(x)\mu_{\alpha,\beta-m_1}(x)dx\\
&=\int J_n^{\alpha,\beta}(x)\mu_{\alpha,\beta-m_1+l-1}(x)dx\\
&=\frac{2^{\alpha+\beta-m_1+l}\Gamma (\beta+1)\Gamma(\beta-m_1+l)\Gamma^{\alpha,\alpha+\beta}_{\beta,\alpha+\beta+l-m_1}(n)} {\Gamma(\alpha+\beta+1)}\binom{n+m_1-l}{n},
\end{align*}
(where we are using the notation (\ref{gammas})), and for $l=m_1+1,\ldots, m$,
\begin{align*}
w_{n,l-m_1-1}^2&=\int (1-x)^{l-m_1-1}J_n^{\alpha,\beta}(x)\mu_{\alpha-m_2,\beta}(x)dx\\
&=\int J_n^{\alpha,\beta}(x)\mu_{\alpha+l-m-1,\beta}(x)dx\\
&=\frac{(-1)^n2^{\alpha+\beta-m+l}\Gamma(n+\alpha+\beta+1)\Gamma(\beta+1)\Gamma(\alpha-m+l)}{\Gamma(\alpha+\beta+1)\Gamma(\alpha+\beta+n+l-m+1)}\binom{n+m-l}{n}.
\end{align*}
We also need the following identities (after a combination of formulas (3.94), (3.100) and (3.107) of \cite{STW}):
\begin{align*}
\left(J_n^{\alpha,\beta}\right)^{(i)}(-1)&=\frac{(-1)^{i}i!}{2^i\binom{\alpha+\beta}{\beta}}\binom{n+\alpha+\beta}{\alpha}\binom{n+\beta}{n-i}\binom{n+\alpha+\beta+i}{i},\\
\left(J_n^{\alpha,\beta}\right)^{(i)}(1)&=\frac{(-1)^ni!}{2^i\binom{\alpha+\beta}{\beta}}\binom{n+\alpha+\beta}{\alpha}\binom{n+\alpha}{n-i}\binom{n+\alpha+\beta+i}{i}.
\end{align*}
If we replace these identities in \eqref{RdSip} and \eqref{RdSip2}, we get, after straightforward computations, the expressions
\begin{align}\label{ri12}
R_l(n)&=\frac{\Gamma(\beta+1)\Gamma(n+\alpha+1)}{\Gamma(\alpha+\beta+1)\Gamma(n+\beta+1)}\left[\frac{2^{\alpha+\beta-m_1+l}\Gamma(\beta-m_1+l)}{(m_1-l)!} \Gamma^{m_1-l,\alpha+\beta}_{0,\alpha+\beta-m_1+l}(n)\right.\\
\nonumber &\qquad\qquad\qquad\left.+2^{m_2}\sum_{i=0}^{m_1-1}\left(\sum_{j=l}^{(l+m_2)\wedge m_1}\frac{(j-1)!\binom{m_2}{j-l}M_{i,j-1}}{(-2)^{i+j-l}}\right)
\frac{\Gamma^{\beta,\alpha+\beta+i}_{-i,\alpha}(n)}{\Gamma(\beta+i+1)}\right],
\end{align}
for $l=1,\ldots, m_1$, and for $l=m_1+1,\ldots, m$,
\begin{align}\label{ri22}
R_l(n)&=\frac{(-1)^n\Gamma(\beta+1)}{\Gamma(\alpha+\beta+1)}\left[\frac{2^{\alpha+\beta-m+l}\Gamma(\alpha-m+l)}{(m-l)!}
\Gamma^{m-l,\alpha+\beta}_{0,\alpha+\beta-m+l}(n)\right.\\
\nonumber &\qquad\qquad \left.+\sum_{i=0}^{m_2-1}\left(\sum_{j=l-m_1}^{l\wedge m_2}\frac{(j-1)!\binom{m_1}{l-j}N_{i,j-1}}{(-1)^{l-m_1-1}2^{i+j-l}}\right)
\frac{\Gamma^{\alpha,\alpha+\beta+i}_{-i,\beta}(n)}{\Gamma(\alpha+i+1)}\right].
\end{align}
Comparing (\ref{ri12}) and (\ref{ri22}) with (\ref{ri1}) and (\ref{ri2}), we get straightforwardly
\begin{equation*}
R_l(n)=\begin{cases}
\displaystyle\frac{\Gamma(\beta+1)\Gamma(n+\alpha+1)}{\Gamma(\alpha+\beta+1)\Gamma(n+\beta+1)}z_l(n),&\quad l=1,\ldots, m_1,\\
\displaystyle\frac{(-1)^n\Gamma(\beta+1)}{\Gamma(\alpha+\beta+1)}z_l(n),&\quad l=m_1+1,\ldots, m.
\end{cases}
\end{equation*}
And therefore
\begin{equation*}
R_l(n-j)=\begin{cases}
\displaystyle\frac{(-1)^m\Gamma(\beta+1)\Gamma(n+\alpha+1-m)}{\Gamma(\alpha+\beta+1)\Gamma(n+\beta)}(-1)^j\no_{n,j}^lz_l(n-j),&\quad l=1,\ldots, m_1,\\
\displaystyle\frac{(-1)^n\Gamma(\beta+1)}{\Gamma(\alpha+\beta+1)}(-1)^j\no_{n,j}^lz_l(n-j),&\quad l=m_1+1,\ldots, m,
\end{cases}
\end{equation*}
where $\no_{n,j}^l, l=1,\ldots, m,$ was defined by \eqref{defno}.

Since $n\ge m$, the hypothesis on $\alpha$ and $\beta$ shows that $n+\alpha+1-m,n+\beta\not =0,-1,-2,\ldots$ and hence Theorem \ref{LemJacSob} gives that the polynomial $q_n$, $n\ge m$, is orthogonal with respect to the Jacobi Sobolev inner product defined by (\ref{idslip}).

\medskip
\noindent
\textit{Second step. Assume $n=0,1,\ldots, m-1$.}
When $\alpha$ and $\beta$ are integers, $p(n)q(n)$ can vanish for some $n=0,\ldots, m-1$. Hence, we first prove that even if for some $n=0,\ldots, m-1$, $p(n)q(n)=0$, the ratio $\Lambda_n$ (\ref{casdet}) and the polynomial $q_n$ (\ref{iquss}) are well defined (and hence $q_n$ has degree $n$ if and only if $\Lambda_n\not =0$).

Indeed, assume first that $p(n)=0$. From \eqref{def1pi} we have that $n$ should be either $n=-\alpha+m-i$ or $n=-\beta +i$, $i=1,\ldots , m_1-1$. Consider first the case $n=-\alpha+m-i$. Write $\tilde \Lambda $ for the $m\times (m+1)$ matrix function
\begin{equation*}\label{casdet2}
\tilde \Lambda (x)=\begin{pmatrix}
\no^1_{x,0} z_1(x) &\no^1_{x,1} z_1(x-1) & \cdots &  \no^1_{x,m}z_1(x-m) \\
               \vdots & \vdots & \ddots & \vdots \\
\no^m_{x,0} z_m(x) & \no^m_{x,1} z_m(x-1) & \cdots & \no^m_{x,m}z_m(x-m)
             \end{pmatrix}.
\end{equation*}
For $j\ge 1$, the function $\rho^h_{x,j}$, $h=1,\ldots, m_1$, (\ref{defno}) is actually a polynomial: $\no^h_{x,j}=(-1)^{m-j}(x+\alpha-m+1)_{m-j}(x+\beta-j+1)_{j-1}$. Hence $n=-\alpha+m-i$ is  a root of $\no^h_{x,j}$, for $j=0,\ldots , m-i$. It is then easy to see that $\rank \tilde \Lambda(-\alpha+m-i)\le m_2+i$. So for $h=1,\ldots , m+1$, $0$ is an eigenvalue of $\tilde \Lambda_h(-\alpha+m-i)$ of geometric multiplicity at least $m-m_2-i=m_1-i$, where $\tilde \Lambda_h$ is the square matrix obtained by removing the $h$-th column of $\tilde \Lambda$. We then deduce that $0$ is an eigenvalue of $\tilde \Lambda_h(-\alpha+m-i)$ of algebraic multiplicity at least $m_1-i$. This implies that $x=-\alpha+m-i$ is a root of $\det \tilde \Lambda_h(x)$ of multiplicity at least $m_1-i$, which it is precisely the multiplicity of $-\alpha+m-i$ as a zero of $p(x)$. A similar result can be proved for the other zeros of $p$ and the zeros of $q$. This proves that the ratios
$\det \tilde \Lambda_h(x)/(p(x)q(x))$ are well defined even when $p(x)q(x)=0$. Hence for $n=0,\ldots, m-1$,
the ratio  $\Lambda_n$ (\ref{casdet}) and the polynomial $q_n$ (\ref{iquss}) are well defined and $q_n$ has degree $n$ if and only if $\Lambda_n\not =0$.

We now prove that $q_n$ are orthogonal for $n=0,1,\ldots,m-1$. We
need first to introduce  some notation. For $l=1,\ldots, m$, we
define
\begin{equation}\label{defll}
\lambda _l(n)=\begin{cases}\displaystyle \frac{(-1)^m\Gamma(\alpha+\beta+1)\Gamma(\beta+n)}{\Gamma(\beta+1)\Gamma(\alpha+n+1-m)}, & \quad l=1,\ldots , m_1,\\\displaystyle\frac{(-1)^n\Gamma(\alpha+\beta+1)}{\Gamma(\beta+1)}, & \quad l=m_1+1,\ldots , m.
\end{cases}
\end{equation}
Write $F_n(x)$, $f_l(n)$, $l=1,\ldots , m$, for the row vectors of size $m+1$ whose entries are
\begin{equation}\label{defff}
F_{n,j}(x)=(-1)^{j-1}J_{n-j+1}^{\alpha,\beta}(x),\quad f_{l,j}(n)=\rho^l_{n,j-1}z_l(n-j+1)
\end{equation}
for $j=1,\ldots, m+1$ (in particular, $F_{n,j}=0$, $j=n+2,\ldots, m+1$). Hence
$$
q_n(x)=\frac{1}{p(n)q(n)}\begin{vmatrix}F_{n}(x)\\
f_1(n)\\\vdots \\f_m(n)\end{vmatrix}.
$$
Consider the basis  of $\PP _{n-1}$ given by $v_h(x)=(1-x)^h$, $h=0,\ldots , n-1$. Since $\deg(v_h)=h$, it is enough to prove that $\langle q_n,v_h\rangle=0$,
$h=0,\ldots , n-1$. Since $b_l(x)$, $l=1,\ldots m$, (see (\ref{base})) is a basis of $\PP_{m-1}$ and $n-1\le m-1$, we have $v_h(x)=\sum_{l=1}^ma_{h,l}b_l(x)$, and then
\begin{equation}\label{defqv}
\langle q_n,v_h\rangle=\frac{1}{p(n)q(n)}\begin{vmatrix}\left(\displaystyle\sum_{l=1}^ma_{h,l}\langle F_{n,j},b_l\rangle,j=1,\ldots ,m+1\right)\\
f_1(n)\\\vdots \\f_m(n)\end{vmatrix}.
\end{equation}
Notice that $p(n)q(n)$ and each entry of the rows $f_l(n)$
(\ref{defff}) are meromorphic functions of $\alpha$ or $\beta$. It was shown in the first step that
$\int_{-1}^1J_n^{\alpha,\beta}(x)b_l(x)d\mu_{\alpha-m_2,\beta-m_1}(x)dx$
and $(J_n^{\alpha,\beta})^{(i)}(\pm 1)$ are also meromorphic
functions of $\alpha$ or $\beta$, and then $\langle
F_{n,j},b_l\rangle$ is also meromorphic. This shows that $\langle
q_n,v_h\rangle$ is a meromorphic function of $\alpha$ or $\beta$.
It is then enough to prove that $\langle q_n,v_h\rangle=0$
assuming that $\alpha$, $\beta$ and $\alpha+\beta$ are
non-integers real numbers. Hence we have
$p(n)q(n),\lambda_l(n)\not =0$, $l=1,\ldots, m$ (\ref{defll}).

Proceeding as in the proof of the first step, we can prove that if $\deg(b_l(x))<n$ then
\begin{equation}\label{j1}
f_{l,j}(n)=\lambda_l(n)\langle F_{n,j},b_l\rangle,\quad j=1,\ldots, m+1.
\end{equation}
On the other hand, if $\deg(b_l(x))\ge n$ then
\begin{equation}\label{j2}
f_{l,j}(n)=\lambda_l(n)\langle F_{n,j},b_l\rangle,\quad j=1,\ldots, n+1.
\end{equation}
Write $u=(\sum_{l=1}^ma_{h,l}\langle F_{n,j},b_l\rangle,j=1,\ldots ,m+1)$, i.e., $u$ is the first row in the determinant (\ref{defqv}). Since $p(n)q(n),\lambda_l(n)\not =0$, $\langle q_n,v_h\rangle=0$ will follow if we prove that $u=\sum_{l=1}^m\frac{a_{h,l}}{\lambda_l(n)}f_l(n)$, i.e. a linear combination of the rest of the rows.

For $j=1,\ldots , n+1$, we  have from (\ref{j1}) and (\ref{j2})
$$
u_j=\sum_{l=1}^ma_{h,l}\langle F_{n,j},b_l\rangle=\sum_{l=1}^m\frac{a_{h,l}}{\lambda_l(n)}f_{l,j}(n).
$$
For $j=n+2,\ldots , m+1$, $F_{n,j}=0$ (\ref{defff}) and then $u_j=0$. Taking into account the definition of $f_l$ (\ref{defff}), it is then enough to prove that
\begin{equation}\label{rey}
\sum_{l=1}^m\frac{a_{h,l}}{\lambda_l(n)}\rho_{n,j-1}^lz_l(n-j+1)=0,
\end{equation}
for $h=0,\ldots, n-1$, $j=n+2,\ldots , m+1$ and $n=0,\ldots, m-1$.

Since $n-j+1<0$, we have from (\ref{ri1}) and (\ref{ri2})
$$
z_l(n-j+1)=\begin{cases} \displaystyle\frac{\Gamma(\beta-m_1+l)\Gamma^{m_1-l,\alpha+\beta}_{0,\alpha+\beta-m_1+l}(n-j+1)}{2^{-\alpha-\beta+m_1-l}(m_1-l)!}, & \quad l=1,\ldots, m_1,\\
\displaystyle\frac{\Gamma(\alpha-m+l)\Gamma^{m-l,\alpha+\beta}_{0,\alpha+\beta-m+l}(n-j+1)}{2^{-\alpha-\beta+m-l}(m-l)!}
,&\quad l=m_1+1,\ldots, m.
\end{cases}
$$
Using (\ref{defno}) and (\ref{defff}), we get
$$
\frac{\rho_{n,j-1}^lz_l(n-j+1)}{\lambda_l(n)}=\begin{cases} c_n\displaystyle\frac{\Gamma(\beta-m_1+l)\Gamma^{\alpha,m_1-l}_{0,\alpha+\beta-m_1+l}(n-j+1)}
{(-1)^{j-1}2^{m_1-l}(m_1-l)!},&\quad l=1,\ldots, m_1,\\
c_n\displaystyle\frac{\Gamma(\alpha-m+l)\Gamma^{\beta,m-l}_{0,\alpha+\beta-m+l}(n-j+1)}
{(-1)^n2^{m-l}(m-l)!},&\quad l=m_1+1,\ldots, m,
\end{cases}
$$
where $\displaystyle c_n=\frac{2^{\alpha+\beta}\Gamma(\beta+1)\Gamma(n-j+\alpha+\beta+2)}{\Gamma(\alpha+\beta+1)\Gamma(n-j+\beta+2)}$.
Inserting them into (\ref{rey}), we get
\begin{align}\label{bor}
(-1)^{j-1}&\sum_{l=1}^{m_1}\frac{a_{h,l}\Gamma(\beta-m_1+l)\Gamma^{\alpha,m_1-l}_{0,\alpha+\beta-m_1+l}(n-j+1)}
{2^{m_1-l}(m_1-l)!}\\\nonumber
&+(-1)^{n}\sum_{l=1}^{m_2}\frac{a_{h,m_1+l}\Gamma(\alpha-m_2+l)\Gamma^{\beta,m_2-l}_{0,\alpha+\beta-m_2+l}(n-j+1)}
{2^{m_2-l}(m_2-l)!}=0.
\end{align}
If we write $k=-n+j-1$, after straightforward computations we can rewrite (\ref{bor}) in the simpler form
\begin{equation}\label{sfo}
(-1)^k\sum_{l=0}^{m_1-1}\frac{a_{h,m_1-l}\binom{l-k}{l}}{2^l(\beta-l)\binom{\alpha+\beta-k-l}{\alpha-k}}
+\sum_{l=0}^{m_2-1}\frac{a_{h,m-l}\binom{l-k}{l}}{2^l(\alpha-l)\binom{\alpha+\beta-k-l}{\beta-k}}=0,
\end{equation}
for $k=1,\ldots,m-h-1$, $h=0,\ldots, m-2$. This identity is then equivalent to (\ref{rey}). Then we finish the proof by proving (\ref{sfo}). For $h=0$, it is easy to check that
$$
a_{0,l}=\begin{cases}\displaystyle\frac{\displaystyle\binom{l+m_2-2}{m_2-1}}{2^{l+m_2-1}},&\quad l=1,\ldots, m_1,\\
\displaystyle\frac{\displaystyle\binom{l-2}{l-m_1-1}}{2^{l-1}},&\quad l=m_1+1,\ldots, m.
\end{cases}
$$
Inserting them into (\ref{sfo}), we get the identity (\ref{comb2}) in Lemma \ref{comb}.

The rest of the proof proceeds by induction on $m_2$. We first consider $m_2=0$, for which $b_l(x)=(1+x)^{l-1}$, $l=0,\ldots , m_1$. A simple computation gives
$$
a_{h,l}=(-1)^l2^{h-l}\binom{h}{l}.
$$
Inserting this identity in (\ref{sfo}), we get
$$
\sum_{l=0}^{m_1-1}\frac{(-1)^l\binom{h}{m_1-l}\binom{l-k}{l}}{2^l(\beta-l)\binom{\alpha+\beta-k-l}{\alpha-k}}=0.
$$
But this is the identity (\ref{comb1}) in Lemma \ref{comb}.

From now on, we write $b_l^{m_2}$, $l=1,\ldots, m$, for the basis (\ref{base}) corresponding to the nonnegative integers $m_1$ and $m_2$, and $b_l^{m_2+1}$, $l=1,\ldots, m+1$, for the basis (\ref{base}) corresponding to $m_1$ and $m_2+1$. We also write
$a_{h,l}^{m_2}$, $a_{h,l}^{m_2+1}$ for the corresponding coefficients of $v_h$ with respect to $b_l^{m_2}$, $b_l^{m_2+1}$, respectively.
For $h=1,\ldots , m+1$, we straightforwardly have the following relationship between $a_{h,l}^{m_2}$ and $a_{h,l}^{m_2+1}$
$$
a_{h,l}^{m_2+1}=\begin{cases} a_{h-1,l}^{m_2},& l=1,\ldots, m_1,\\
0,& l=m_1+1,\\
a_{h-1,l-1}^{m_2}, &l=m_1+2,\ldots, m+1.
\end{cases}
$$
This shows that the identity (\ref{sfo}) for $m_2+1$ and $h$ reduces to (\ref{sfo}) for $m_2$ and $h-1$, and hence the induction hypothesis says that (\ref{sfo}) holds for $m_2+1$ and $h=1,\ldots , m+1$.
\end{proof}

\section{$\D$-operators}\label{sec3}
The $\D$-operator concept was introduced by one of us in \cite{du1}. In  \cite{AD}, \cite{du1}--\cite{ddI2}, it was shown that $\D$-operators are an extremely useful tool of an unified method for generating families of polynomials that are eigenfunctions of higher-order differential, difference or $q$-difference operators. Hence, we start by recalling the concept of $\D$-operator.

The starting point is a sequence of polynomials $(p_n)_n$, $\deg(p_n)=n$, and an algebra of operators $\A $ that act in the linear space of
polynomials $\mathbb{P}$. For the Jacobi polynomials we consider the algebra $\A $ formed by all differential operators of finite order which do not increase the degree of polynomials, i.e.
\begin{equation}\label{algdiffc}
\A =\left\{ \sum_{j=0}^sf_j\left(\frac{d}{dx}\right)^j : f_j\in \PP, \deg(f_j)\le j, j=0,\ldots,s, s\in \NN \right\}.
\end{equation}
If $f_s\neq0$ then the \emph{order} of such differential operator is $s$. In addition, we assume that the polynomials $p_n$, $n\ge 0$, are eigenfunctions of a certain operator $D_p\in \A$. We write $(\theta_n)_n$ for the corresponding eigenvalues such that $D_p(p_n)=\theta_np_n$, $n\ge 0$. For the Jacobi polynomials, $\theta_n$ is a polynomial in $n$ of degree 2 (see (\ref{autov})), but we do not assume any constraint on the sequence $(\theta_n)_n$ in this section.

Given two sequences of numbers, $(\varepsilon_n)_n$ and $(\sigma_n)_n$, a $\D$-operator associated with the algebra $\A$ and the sequence of polynomials $(p_n)_n$ is defined as follows. First, we consider  the operator $\D :\PP \to \PP $ defined by linearity from
\begin{equation*}\label{defTo}
\D (p_n)=-\frac{1}{2}\sigma_{n+1}p_n+\sum _{j=1}^n (-1)^{j+1}\sigma_{n-j+1}\varepsilon _n\cdots \varepsilon _{n-j+1}p_{n-j},\quad n\ge 0.
\end{equation*}
Then, we say that $\D$ is a $\D$-operator if $\D\in \A$. In \cite{du1} this type of $\D$-operator was designated as type 2, whereas $\D$-operators of type 1 appear when the sequence $(\sigma_n)_n$ is constant. $\D$-operators of type 1 are simpler but they are only useful when the sequence of eigenvalues $(\theta_n)_n$ is linear in $n$; this is the reason why we used $\D$-operators of type 1 in \cite{ddI1} for discrete Laguerre-Sobolev
polynomials, but we have to use $\D$-operators of type 2 in this paper for discrete Jacobi-Sobolev polynomials.

Let us now provide a couple of examples of $\D$-operators for the Jacobi polynomials. We now consider the algebra $\A$ of differential operators defined by \eqref{algdiffc}. The two $\D$-operators for the Jacobi polynomials are defined by the sequences $(\varepsilon_{n,h})_n$ and $(\sigma_{n,h})_n, h=1,2,$ given by
\begin{align}
\label{eps1}&\varepsilon_{n,1}=-\frac{n+\alpha}{n+\beta},\quad&\sigma_{n,1}&=\sigma_n=2n+\alpha+\beta-1,\\
\label{eps2}&\varepsilon_{n,2}=1,\quad&\sigma_{n,2}&=-\sigma_n=-(2n+\alpha+\beta-1).
\end{align}
As proved in Lemma A.7 of \cite{du1}, these sequences define two  $\mathcal{D}$-operators $\mathcal{D}_1$
and $\mathcal{D}_2$ for the Jacobi polynomials. Moreover
\begin{equation}\label{dopjac}
\mathcal{D}_1=-\frac{\alpha+\beta+1}{2}I+(1-x)\frac{d}{dx},\quad\mbox{and}\quad\mathcal{D}_2=\frac{\alpha+\beta+1}{2}I+(1+x)\frac{d}{dx}.
\end{equation}

We next show how to use $\D$-operators to construct new sequences of polynomials $(q_n)_n$ such that there exists an operator $D_q\in \A$ for which they are eigenfunctions (we follow the same lines as Section 3 in \cite{ddI2}).

Consider a combination of $m+1$, $m\ge 1$, consecutive $p_n$'s.
We also use $m$ arbitrary polynomials $Y_1, Y_2, \ldots, Y_m,$ and $m$ $\D$-operators, $\D_1, \D_2, \ldots, \D_m,$ (which are not necessarily different) defined by the pairs of sequences $(\varepsilon_n^h)_n$, $(\sigma_n^h)_n$, $h=1,\ldots , m$:
\begin{equation}\label{Dh}
\mathcal{D}_h(p_n)=-\frac{1}{2}\sigma_{n+1}^hp_n+\sum _{j=1}^n (-1)^{j+1}\sigma_{n-j+1}^h\varepsilon_{n}^{h}\cdots\varepsilon_{n-j+1}^{h}p_{n-j}.
\end{equation}

For $h=1,2,\ldots,m$, we assume that the sequences $(\varepsilon_{n}^{h})_n$ and $(\sigma_{n}^{h})_n$ are rational functions in $n$. We write $\xi_{x,i}^h$, $i\in\ZZ$ and $h=1,2,\ldots,m$, for the auxiliary functions defined by
\begin{equation}\label{defxi}
\xi_{x,i}^h=\prod_{j=0}^{i-1}\varepsilon_{x-j}^{h}, \quad i\ge 1,\quad \quad \xi_{x,0}^h=1,\quad\quad \xi_{x,i}^h=\frac{1}{\xi_{x-i,-i}^h},\quad i\leq-1.
\end{equation}
We consider the $m\times m$ (quasi) Casorati determinant defined by
\begin{equation}\label{casd1}
\Omega (x)=\det \left(\xi_{x-j,m-j}^lY_l(\theta_{x-j})\right)_{l,j=1}^m.
\end{equation}
Then, we have the following

\begin{theorem}[Theorem 3.1 of \cite{ddI2}]\label{Teor1} Let $\A$ and $(p_n)_n$ be an algebra of operators that act in the linear space of polynomials and a sequence of polynomials $(p_n)_n$, $\deg(p_n)=n$, respectively. We assume that $(p_n)_n$ are eigenfunctions of an operator $D_p\in \A$, i.e., the numbers $\theta_n, n\geq0$, exist such that $D_p(p_n)=\theta_np_n$, $n\ge 0$. We also have $m$ pairs of sequences of numbers $(\varepsilon_n^h)_n$, $(\sigma_n^h)_n$, $h=1,\ldots , m$, which define $m$ $\D$-operators $\D_1,\ldots,\D_m$ (not necessarily different) for $(p_n)_n$ and $\A$ (see (\ref{Dh})) and for $h=1,2,\ldots,m$,  we assume that each one of the sequences $(\varepsilon_{n}^{h})_n$, $(\sigma_n^h)_n$ is a rational function in $n$.

Let $Y_1, Y_2, \ldots, Y_m,$ be $m$ arbitrary polynomials that satisfy $\Omega (n)\not =0$, $n\ge 0$, where $\Omega $ is the Casorati determinant defined by (\ref{casd1}).

Consider the sequence of polynomials $(q_n)_n$ defined by
\begin{equation}\label{qus}
q_n(x)=\begin{vmatrix}
               p_n(x) & -p_{n-1}(x) & \cdots & (-1)^mp_{n-m}(x) \\
               \xi_{n,m}^1Y_1(\theta_{n}) &  \xi_{n-1,m-1}^1Y_1(\theta_{n-1}) & \cdots & Y_1(\theta_{n-m}) \\
               \vdots & \vdots & \ddots & \vdots \\
                \xi_{n,m}^mY_m(\theta_{n}) &  \xi_{n-1,m-1}^mY_m(\theta_{n-1}) & \cdots & Y_m(\theta_{n-m})
             \end{vmatrix}.
\end{equation}
For a rational function $S$, we define the function $\lambda_x$ by
\begin{equation*}\label{deflx}
\lambda_x-\lambda_{x-1}=S(x)\Omega(x),
\end{equation*}
and for $h=1,\ldots,m$, we define the function $M_h(x)$ by
\begin{equation}\label{emeiexp}
M_h(x)=\sum_{j=1}^m(-1)^{h+j}\xi_{x,m-j}^hS(x+j)\det\left(\xi_{x+j-r,m-r}^{l}Y_l(\theta_{x+j-r})\right)_{l\in \II_h;r\in\II_j},
\end{equation}
where $\II_h=\{1,2,\ldots , m\}\setminus \{ h\}$. We assume the following:
\begin{align}\label{ass0}
&\mbox{$S(x)\Omega (x)$ is a polynomial in $x$.}\\\label{ass1}
&\mbox{Polynomials  $\tilde{M}_1,\ldots,\tilde{M}_m,$  exist such that}\\\nonumber
&\hspace{1cm}\mbox{$M_h(x)=\sigma_{x+1}^h\tilde{M}_h(\theta_x)$, $h=1,\ldots,m.$}\\\label{ass2}
&\mbox{A polynomial $P_S$ exists such that $\displaystyle P_S(\theta_x)=2\lambda_x+\sum_{h=1}^mY_h(\theta_x)M_h(x)$.}
\end{align}
Then, an operator $D_{q,S}\in \A$ exists such that
$$
D_{q,S}(q_n)=\lambda_nq_n,\quad n\ge 0.
$$
Moreover, the operator $D_{q,S}$ is defined by
\begin{equation}\label{Dq}
D_{q,S}=\frac{1}{2}P_S(D_p)+\sum_{h=1}^m\tilde M_h(D_p)\D_hY_h(D_p),
\end{equation}
where $D_p\in \A$ is the operator for which the polynomials $(p_n)_n$ are eigenfunctions.
\end{theorem}

\begin{remark}
For the particular cases of Laguerre, Jacobi, or Askey-Wilson polynomials, we can find Casorati determinants similar to (\ref{qus}) in \cite{GrHH}--\cite{Plamen3}.
\end{remark}

\begin{remark}
According to Remark 3.2 in \cite{ddI2}, the polynomial $P_S$ (\ref{ass2}) also satisfies
\begin{equation*}\label{Pdiff}
P_S(\theta_x)-P_S(\theta_{x-1})=S(x)\Omega(x)+S(x+m)\Omega(x+m).
\end{equation*}
\end{remark}

\bigskip
The assumptions (\ref{ass0}), (\ref{ass1}) and (\ref{ass2}) turn out to be straightforward for $\D$-operators of type 1 (we can then take $\theta _x=x$) but they need to be checked when we use $\D$-operators of type 2. The rest of this section will be devoted to check these three assumptions for the $\D$-operators (\ref{dopjac}) associated to the Jacobi polynomials.

We need to introduce some notation. We write $N_{1;x}^{\alpha;j}$ and $N_{2;x}^{\beta;j}$, $j\in \NN$ and $x\in \RR$, for the following functions:
\begin{equation*}\label{ABDijH}
N_{1;x}^{\alpha;j}=(-1)^j(x-j+\alpha+1)_j,\quad N_{2;x}^{\beta;j}=(x-j+\beta+1)_j.
\end{equation*}
The following properties hold easily by definition
\begin{equation*}\label{ABprop2}
N_{1;x-i}^{\alpha;m-i}=N_{1;x-i}^{\alpha;j}N_{1;x-i-j}^{\alpha;m-i-j},\quad N_{2;x-i}^{\beta;m-i}=N_{2;x-i}^{\beta;j}N_{2;x-i-j}^{\beta;m-i-j}.
\end{equation*}
Assume now that, as in Theorem \ref{Teor1}, we have $m$ $\D$-operators $\D_i$, $i=1,\ldots, m$, associated to the Jacobi polynomials
and defined by the sequences $(\varepsilon_n^h)_n$ and $(\sigma_n^h)_n$, $h=1,\ldots, m$.
Assume also that they correspond with the two $\D$-operators defined by the sequences (\ref{eps1}) and (\ref{eps2}). More precisely,
\begin{equation}\label{cons}
\varepsilon_n^h=\begin{cases} \varepsilon_{n,1}, & \mbox{for $h=1,\ldots, m_1$},\\
\varepsilon_{n,2}, & \mbox{for $h=m_1+1,\ldots, m$},
\end{cases}
\quad
\sigma_n^h=\begin{cases} \sigma_{n,1}, & \mbox{for $h=1,\ldots, m_1$},\\
\sigma_{n,2}, & \mbox{for $h=m_1+1,\ldots, m$}.
\end{cases}
\end{equation}
where the sequences $(\varepsilon_{n,1})_n$, $(\sigma_{n,1})_n$, and $(\varepsilon_{n,2})_n$, $(\sigma_{n,2})_n$, are defined by (\ref{eps1}) and (\ref{eps2}), respectively. The functions $\xi_{x,j}^h$ defined in \eqref{defxi} can then be written as
\begin{equation}\label{xit}
\xi_{x,j}^h=\begin{cases} \displaystyle\frac{N_{1;x}^{\alpha;j}}{N_{2;x}^{\beta;j}}=\frac{(-1)^j(x-j+\alpha+1)_j}{(x-j+\beta+1)_j},& \mbox{for $h=1,\ldots, m_1$,}\\
1,& \mbox{for $h=m_1+1,\ldots, m$.}
 \end{cases}
\end{equation}
We also need to introduce the polynomials $p$ and $q$ defined by
\begin{align}\label{def1p}
p(x)&=\prod_{i=1}^{m_1-1}N^{\alpha;m_1-i}_{1;x-m_2-i}N^{\beta;m_1-i}_{2;x-1},\\
\label{def1q}q(x)&=(-1)^{\binom{m}{2}}\prod_{h=1}^{m-1}\left(\prod_{i=1}^{h}\sigma_{x-m+\frac{i+h+1}{2}}\right),
\end{align}
where as in (\ref{eps1}) $\sigma_x=2x+\alpha+\beta-1$. It is easy to check that the polynomial $p$  in (\ref{def1p}) is the same as the polynomial $p$ defined in (\ref{def1pi}), as well as the polynomial $q$ in (\ref{def1q}) is the same as the polynomial $q$ defined in (\ref{def1qi}).

The key concept in order to check the assumptions (\ref{ass0}), (\ref{ass1}) and (\ref{ass2}) in Theorem \ref{Teor1} for the Jacobi polynomials is an \emph{involution} that characterizes the subring $\mathbb{R}[\theta_x]$ in $\mathbb{R}[x]$, where $\theta _x=x(x+\alpha+\beta+1)$ are the eigenvalues for the Jacobi polynomials. This involution is given by
\begin{equation}\label{inv}
\I^{\alpha+\beta}\big(f(x)\big)=f\big(-(x+\alpha+\beta+1)\big),\quad f\in\mathbb{R}[x].
\end{equation}
Clearly, we have $\I^{\alpha+\beta}(\theta_x)=\theta_x$. Hence every polynomial in $\theta_x$ is invariant under the action of $\I^{\alpha+\beta}$. Conversely, if $f\in\mathbb{R}[x]$ is invariant under $\I^{\alpha+\beta}$, then $f\in\mathbb{R}[\theta_x]$.

We also have that if $f\in\mathbb{R}[x]$ is skew invariant, i.e., $\I^{\alpha+\beta}(f)=-f,$ then $f$ is divisible by $\theta_{x-1/2}-\theta_{x+1/2}$ and the quotient belongs to $\mathbb{R}[\theta_x]$. We remark that in the case of Jacobi polynomials and from the definition of $\theta_x$ and $\sigma_x$ we have that $\sigma_{x+1}=\theta_{x-1/2}-\theta_{x+1/2}$. We observe that $\sigma_{x+1}$ is itself skew invariant.

According to the definition \eqref{inv} we have the following properties:
\begin{align}
\label{Iprop1}\I^{\alpha+\beta+i}(\theta_{x-j})&=\theta_{x+i+j},\quad &\I^{\alpha+\beta+i}(\sigma_{x-j})&=-\sigma_{x+i+j+2},\\
\label{ABpropsH}
\I^{\alpha+\beta+i}(N_{1;x-j-h}^{\alpha;m-h})&=N_{2;x+m+i+j}^{\beta;m-h},&\quad \I^{\alpha+\beta+i}(N_{2;x-j-h}^{\beta;m-h})&=N_{1;x+m+i+j}^{\alpha;m-h}.
\end{align}

\medskip

We are now ready to establish that the three assumptions (\ref{ass0}), (\ref{ass1}) and (\ref{ass2}) in Theorem \ref{Teor1}
hold for the two $\D$-operators associated with the Jacobi polynomials.

\begin{lemma}\label{l5.2} Let $\A$ and $(p_n)_n$ be the algebra of differential operators (\ref{algdiffc}) and the sequence of Jacobi polynomials $p_n=J_n^{\alpha,\beta}$, respectively. Let $D_{\alpha,\beta}$ be the second-order differential operator (\ref{dopJac}) such that $\theta_n=n(n+\alpha+\beta+1)$ and
$D_{\alpha,\beta}(J_n^{\alpha,\beta})=\theta_nJ_n^{\alpha,\beta}$. For $j=1,2$, we also have $m_j$ $\D$-operators defined by the sequences $(\varepsilon_{n,j})_n$, $(\sigma_{n,j})_n$ (see
(\ref{eps1}) and (\ref{eps2})). Then, we write $m=m_1+m_2$ and we let $\Xi$ be a  polynomial in $x$, which is invariant under the action of $\I^{\alpha+\beta-m-1}$. We define  the rational function $S$  by
\begin{equation}\label{GGH}
S(x)=\frac{\sigma_{x-\frac{m-1}{2}}\Xi(x)\left(N_{2;x-1}^{\beta;m-1}\right)^{m_1}}{p(x)q(x)},
\end{equation}
where $p$ and $q$ are the polynomials defined by (\ref{def1p}) and (\ref{def1q}), respectively. Then, the three assumptions (\ref{ass0}), (\ref{ass1}), and (\ref{ass2}) in Theorem \ref{Teor1} hold for any polynomials $Y_l$, $l=1,\ldots, m$.
\end{lemma}

The proof is quite technical at certain points and is given separately in the Appendix.

\section{Differential properties for the discrete Jacobi-Sobolev polynomials}\label{sec4}
In this section we will study differential properties of orthogonal polynomials with respect to the discrete Jacobi-Sobolev bilineal form (\ref{idslip}). They are a consequence of the determinantal representation (\ref{iquss}) and Theorem \ref{Teor1}.

Comparing (\ref{defno})  with (\ref{xit}), we get straightforwardly
\begin{equation*}
\no^h _{n,j}=\begin{cases}
N^{\alpha,m-j}_{1;n-j}N^{\beta;j-1}_{2;n-1}=(\beta+n-m+1)_{m-1}\xi^h_{n-j,m-j},&\quad l=1,\ldots, m_1,\\
\xi^h_{n-j,m-j},&\quad l=m_1+1,\ldots, m.
\end{cases}
\end{equation*}
Theorem \ref{mainth1} gives then the following determinantal expression for the orthogonal polynomials with respect to the discrete Jacobi-Sobolev bilineal form (\ref{idslip})
\begin{equation}\label{iquss2}
q_n(x)=\frac{(\beta+n-m+1)_{m-1}^{m_1}}{p(n)q(n)}\begin{vmatrix}
               J_{n}^{\alpha,\beta}(x) & -J_{n-1}^{\alpha,\beta}(x) & \cdots & (-1)^mJ_{n-m}^{\alpha,\beta}(x) \\
\xi_{n,m}^1 z_1(n) &\xi_{n-1,m-1}^1 z_1(n-1) & \cdots &  z_1(n-m) \\
               \vdots & \vdots & \ddots & \vdots \\
\xi_{n,m}^m z_m(n) &  \xi_{n-1,m-1}^m z_m(n-1) & \cdots & z_m(n-m)
             \end{vmatrix}.
\end{equation}

Observe that the functions $z_l, l=1,\ldots,m,$ (see (\ref{ri1}) and (\ref{ri2})) are not polynomials in $\theta _x$ (not even polynomials in $x$) as they should be if we want to apply Theorem \ref{Teor1}. But it turns out that if $\alpha$ and $\beta$ are nonnegative integers satisfying $\alpha\geq m_2$ and $\beta\geq m_1$, then the functions $z_l, l=1,\ldots,m,$ are polynomials in $\theta_x$. Using that we prove the following

\begin{theorem}\label{mainth3}
Assume that any of the two equivalent properties (1) and (2) in
Theorem \ref{mainth1} hold. If $\bm M ,\bm N \not =0$, we assume, in addition, that $\alpha$ and $\beta$ are nonnegative integers with $\alpha \ge m_2$ and $\beta \ge m_1$. If, instead, $\bm M=0$,  we assume that only $\alpha$ is a positive integer with $\alpha \ge m_2$, and
if $\bm N=0$, we assume that only $\beta$ is a positive integer with $\beta \ge m_1$.
Consider a polynomial $\Xi$ invariant under the action of $\I^{\alpha+\beta-m-1}$ (see (\ref{inv})) and the associated rational function $S$ (see (\ref{GGH})). Then there exists a finite order differential operator $D_{S}$ (which can be constructed using (\ref{Dq}))
for which the orthogonal polynomials $(q_n)_n$
(\ref{iquss}) are eigenfunctions. Moreover, up to an additive constant, the corresponding eigenvalues
$(\lambda_n)_n$ of $D_S$ are $\lambda _n=P_S(\theta_n)$, where $P_S$ is
the polynomial defined by the difference equation
\begin{equation}\label{Pdiff2}
P_S(\theta_x)-P_S(\theta_{x-1})=S(x)\Omega (x)+S(x+m)\Omega (x+m),
\end{equation}
where $\Omega$ is the (quasi) Casorati determinant
\begin{equation}\label{defOmb}
\Omega (x)=\det \left(\xi_{x-j,m-j}^lz_l(x-j)\right)_{l,j=1}^m.
\end{equation}
Moreover, the order of the differential operator $D_S$ is $$\deg \Xi +2(\bwr (\bm M)+\awr (\bm N)+1),$$ where $\awr$ and $\bwr$ are
the $\alpha$ and $\beta$ weighted rank introduced in Definition \ref{wrM}.
\end{theorem}

Theorem \ref{mainth2} and Corollary \ref{cori1} in the Introduction are then easy consequences of Theorem \ref{mainth3} (for the particular case of $\Xi=1$).

\begin{proof}
For the two $\D$-operators associated to the Jacobi polynomials,
Lemma \ref{l5.2} guarantees that the assumptions in Theorem \ref{Teor1} hold for each rational function $S$ defined by (\ref{GGH})
and any polynomials $Y_l$, $l=1,\ldots, m$. If we prove that there exist polynomials $Y_l$, $l=1,\ldots, m$, such that $z_l(x)=Y_l(\theta_x)$, where $z_l$ are the functions defined by (\ref{ri1}) and (\ref{ri2}),
the first part of Theorem \ref{mainth3} will follow as a straightforward consequence of the determinantal representation (\ref{iquss2}) and Theorem \ref{Teor1} (possibly after a renormalization constant).

If $\beta$ is a nonnegative integer with $\beta \ge m_1$, we can rewrite the functions $z_l$, $l=1,\ldots, m_1$, in the form (see (\ref{ri1}))
\begin{align}\label{zz1b}
z_l(x)&=\frac{2^{\alpha+\beta-m_1+l}\Gamma(\beta-m_1+l)}{(m_1-l)!} u^\alpha_{m_1-l}(x)\\
\nonumber &\qquad\quad+2^{m_2}\sum_{i=0}^{m_1-1}\left(\sum_{j=l}^{(l+m_2)\wedge m_1}\frac{(j-1)!\binom{m_2}{j-l}M_{i,j-1}}{(-2)^{i+j-l}}\right)
\frac{u^0_{\beta +i}(x)}{\Gamma(\beta+i+1)},
\end{align}
where $u_j^\lambda$, $\lambda \in \RR$ and $j\in \NN$, is the polynomial of degree $2j$ defined by
\begin{equation}\label{bas}
u_j^\lambda(x)=(x+\alpha-\lambda+1)_j(x+\beta+\lambda-j+1)_j.
\end{equation}
Analogously, if $\alpha$ is a nonnegative integer with $\alpha \ge m_2$, we can rewrite the functions $z_l$, $l=m_1+1,\ldots, m$, in the form (see (\ref{ri2}))
\begin{align}\label{zz2b}
z_l(x)&=\frac{2^{\alpha+\beta-m+l}\Gamma(\alpha-m+l)}{(m-l)!}
u^\alpha_{m-l}(x)\\
\nonumber &\qquad\quad +\sum_{i=0}^{m_2-1}\left(\sum_{j=l-m_1}^{l\wedge m_2}\frac{(j-1)!\binom{m_1}{l-j}N_{i,j-1}}{(-1)^{l-m_1-1}2^{i+j-l}}\right)
\frac{u^{\alpha-\beta}_{\alpha +i}(x)}{\Gamma(\alpha+i+1)}.
\end{align}
But it is easy to see that $u_j^\lambda(x)\in\mathbb{R}[\theta_x]$:
$$
u_j^\lambda(x)=\prod_{i=1}^j\left[(\alpha-\lambda+i)(\beta+\lambda-i+1)+\theta_x\right].
$$
Hence if  $\beta$ is a nonnegative integer with $\beta \ge m_1$, for $l=1,\ldots, m_1$, there exists a polynomial $Y_l$, such that $z_l(x)=Y_l(\theta _x)$, and analogously, if $\alpha$ is a nonnegative integer with $\alpha \ge m_2$, for $l=m_1+1,\ldots, m$, there also exists a polynomial $Y_l$, such that $z_l(x)=Y_l(\theta _x)$. This finishes the proof of the first part of the Theorem.

Now we have to prove that the order of $D_S$ is exactly $\deg \Xi +2(\bwr (\bm M)+\awr (\bm N)+1)$. This proof is quite technical at certain points, so it will be given separately in the Appendix.
\end{proof}

\subsection{Examples}
\ 

\noindent 
1. Consider $\bm M=(M_{i.j})_{i,j=0}^{m_1-1}$ and $\bm N=(N_{i.j})_{i,j=0}^{m_2-1}$ in the discrete Jacobi-Sobolev inner product \eqref{idslip} as the symmetric matrices with entries
\begin{equation*}
M_{i,j}=\begin{cases}
M_{i+j},&\quad i+j\leq m_1-1,\\
0,&\quad i+j>m_1-1,
\end{cases}\quad\mbox{and}\quad N_{i,j}=\begin{cases}
N_{i+j},&\quad i+j\leq m_2-1,\\
0,&\quad i+j>m_2-1,
\end{cases}
\end{equation*}
where $M_{m_1-1}, N_{m_2-1}\neq0$. Then the inner product reduces to the inner product defined by the moment functional
\begin{equation}\label{mff}
\mu_{\alpha-m_2,\beta-m_1}(x)+\sum_{i=0}^{m_1-1}M_i\delta_{-1}^{(i)}+\sum_{i=0}^{m_2-1}N_i\delta_{1}^{(i)}.
\end{equation}
We observe that in this case we can calculate directly the degrees of the polynomials $z_l$ defined by \eqref{zz1b} and \eqref{zz2b}. Indeed
$$
\deg z_l=\begin{cases}
2(\beta+m_1-l),&\quad l=1,\ldots,m_1,\\
2(\alpha+m-l),&\quad l=m_1+1,\ldots,m.
\end{cases}
$$
Then the degree of the polynomial $P$ defined by \eqref{def2p} is given by \eqref{ddp}. Using Lemma \ref{l6.3}  and Corollary \ref{cori1} we deduce that the minimal order of the differential operators having the orthogonal polynomials with respect to \eqref{mff} as eigenfunctions is at most $2(m_1\beta+m_2\alpha+1)$. For the case of $m_1=m_2=1$ we recover Koekoeks' result \cite{koekoe2}.
\medskip

\noindent
2. Consider $\bm M=\mbox{diag}\left(M_0,\ldots,M_{m_1-1}\right)$, $M_{m_1-1}\neq0$ and $\bm N=\mbox{diag}\left(N_0,\ldots,N_{m_2-1}\right)$, $N_{m_2-1}\neq0$ in the discrete Jacobi-Sobolev inner product \eqref{idslip}. The degrees of the polynomials $z_l$ in \eqref{zz1b} and \eqref{zz2b} are now given by
$$
\deg z_l=\begin{cases}
2(\beta+m_1-1),&\quad l=1,\ldots,m_1,\\
2(\alpha+m_2-1),&\quad l=m_1+1,\ldots,m.
\end{cases}
$$
In this case we can not apply Lemma \ref{lgp1} to calculate the degree of the polynomial $P$ in \eqref{def2p}. We have to use Definition \ref{wrM} and Corollary \ref{cori1} to calculate the order of the differential operator. But we already know how to calculate $\awr(\bm N)$ and $\bwr(\bm M)$ if $\bm M$ and $\bm N$ are diagonal matrices (see p. 86 of \cite{ddI1}). Call
\begin{align*}
I=\{j: 1\leq j\leq m_1, M_{j-1}=0\},&\quad J=\{j: 1\leq j\leq m_2, N_{j-1}=0\},\\
s=\#\{j: 1\leq j\leq m_1, M_{j-1}\neq0\},&\quad t=\#\{j: 1\leq j\leq m_2, N_{j-1}\neq0\}.
\end{align*}
Then the minimal order of the differential operators having the orthogonal polynomials as eigenfunctions is at most
\begin{equation*}
2\left[t(\alpha-m_2-1)+s(\beta-m_1-1)+2\sum_{i=1}^2\binom{m_i+1}{2}-2\sum_{j\in I}j-2\sum_{j\in J}j+1\right].
\end{equation*}
For the special case of $I=\{1,\ldots,m_1-1\}$ and $J=\{1,\ldots,m_2-1\}$ we have that $s=t=1$. Therefore the order of the differential operator is given by $2(\alpha+\beta+m_1+m_2-1)$ and we recover Bavinck's result \cite{Ba}.
\medskip

\noindent
3. As we mentioned in the Introduction there are some special situations where it is possible to find a differential operator of order lower than the one given by Theorem \ref{mainth3}. In this theorem the differential operator is obtained from the rational function $S$ (see \eqref{GGH}) by taking $\Xi=1$. However for special values of the parameters $\alpha$ and $\beta$ and the matrices $\bm M$ and $\bm N$, a \emph{better} rational function $S$ can be considered satisfying the three assumptions (\ref{ass0}), (\ref{ass1}) and (\ref{ass2}) in Theorem \ref{Teor1} in such a way that the order of the differential operator constructed using this new $S$ is less than $2(\bwr (\bm M)+\awr (\bm N)+1)$. Here we consider a couple of examples of this situation. In both examples we assume that $m_1=m_2$ and $\alpha=\beta\in \NN $, $\alpha\ge m_1$.

\medskip
\noindent
3.1. Take $m_1=m_2=1$. Then the matrices $\bm M$ and $\bm N$ reduce to numbers which, in addition, we assume they are equal, i.e. $\bm M=\bm N>0$. The polynomials $(q_n)_n$ are then orthogonal with respect to the Gegenbauer type positive measure
\begin{equation}\label{wesp}
(1-x^2)^{\alpha-1}\chi_{[-1,1]}+\bm M(\delta _{-1}+\delta_1).
\end{equation}
Our assumptions imply that (see (\ref{xit}))
\begin{equation}\label{xesp}
\xi_{x,j}^1=(-1)^j, \quad \xi_{x,j}^2=1,
\end{equation}
and
\begin{equation}\label{zesp}
z_1(x)=z_2(x)=4^\alpha(\alpha-1)!+\frac{2\bm M(x+1)_{2\alpha}}{\alpha!},
\end{equation}
where $z_1$ and $z_2$ are the polynomials  defined by (\ref{zz1b}) and (\ref{zz2b}), respectively. Hence we have for the polynomials $Y_1$ and $Y_2$ satisfying $Y_1(\theta_x)=z_1(x)$, $Y_2(\theta_x)=z_2(x)$, respectively, that $Y_1=Y_2$ and both have degree exactly $\alpha$. In particular, we have
\begin{equation}\label{oa}
\Omega(x)=-2z_1(x-1)z_1(x-2),
\end{equation}
where $\Omega$ is the determinant defined by (\ref{defOmb}).

The rational function $S$ (see \eqref{GGH}) in Theorem \ref{mainth3} is now a polynomial of degree 1; more precisely $S(x)=-\frac{1}{2}\sigma_{x-1/2}=-(x+\alpha-1)$. For $\Xi=1$, the associated differential operator $D_S$ in Theorem \ref{mainth3} has order $2(\bwr (\bm M)+\awr (\bm N)+1)=4\alpha +2$. However, for this example there is a better choice for the function $S$, in the sense that one can construct from the new $S$ a differential operator of order $2\alpha +2$ for which the orthogonal polynomials $(q_n)_n$ are eigenfunctions. Indeed, consider the rational function
\begin{equation}\label{defnS}
S(x)=\frac{\sigma_{x-1/2}R(x)}{\Omega(x)},
\end{equation}
where $R$ is the polynomial defined by
$$
R(x)=4^{\alpha-1}(\alpha-1)!+\frac{\bm M(x-1)_\alpha (x+\alpha)_\alpha}{2\alpha!}.
$$
We now check the three assumptions (\ref{ass0}), (\ref{ass1}) and (\ref{ass2}) in Theorem \ref{Teor1}. The first assumption (\ref{ass0}) is trivial since $S(x)\Omega(x)=\sigma_{x-1/2}R(x),$ which is obviously a polynomial. A simple computation shows that the polynomial $R$ satisfies the difference equation
$$
2\sigma_{x+1/2}R(x+1)+2\sigma_{x+3/2}R(x+2)=\sigma_{x+1}z_1(x).
$$
Using this difference equation together with (\ref{xesp}), (\ref{zesp}) and (\ref{oa}) one gets
\begin{equation}\label{Mesp}
M_1(x)=M_2(x)=\frac{1}{4}\sigma_{x+1}=\frac{2x+2\alpha+1}{4},
\end{equation}
where $M_1$ and $M_2$ are the functions defined by (\ref{emeiexp}). The second assumption (\ref{ass1}) is now straightforward by taking $\tilde M_1(x)=\tilde M_2(x)=1/4$.

In this case, the difference equation $\lambda_x-\lambda_{x-1}=S(x)\Omega(x)$ can be easily solved to get
\begin{equation}\label{lesp}
\lambda_x=4^{\alpha-1}(\alpha-1)!x(x+2\alpha-1)+\frac{\bm M(x-1)_{2\alpha+2}}{2(\alpha+1)!}.
\end{equation}
If we write $Q(x)=2\lambda_x+z_1(x)M_1(x)+z_2(x)M_2(x)$, (\ref{zesp}), (\ref{Mesp}) and (\ref{lesp})  give
$$
Q(x)=2\cdot4^{\alpha-1}(\alpha-1)!(\theta_x+2\alpha+1)+\frac{\bm M(x+1)_{2\alpha}}{(\alpha+1)!}\left[\theta_x+(\alpha+1)(2\alpha+1)\right].
$$
It is now easy to see that $I^{\alpha+\beta}(Q)=Q$ and hence $Q$ is actually a polynomial in $\theta _x$. That is, there exists a polynomial $P_S$ such that
$P_S(\theta_x)=Q(x)$, and hence the third assumption (\ref{ass2}) holds. Moreover, since $Q$ has degree $2\alpha +2$, the polynomial $P_S$ has degree exactly $\alpha+1$.

Theorem 3.1 gives that the orthogonal polynomials $(q_n)_n$ are eigenfunctions of the differential operator given by $D_{q,S}$ (\ref{Dq}).
Since in this example $P_S$ is a polynomial of degree $\alpha+1$, $\tilde M_1(x)=\tilde M_2(x)=1/4$, $Y_1(x)=Y_2(x)$ is a polynomial of degree
$\alpha$ and the $\D$-operators for the Jacobi polynomials (see (\ref{dopjac})) are both differential operators of order 1, we deduce that the differential operator $D_{q,S}$ (\ref{Dq}) has order equal to $2\alpha +2$. Hence, for this example the rational function $S$ (\ref{defnS}) provides for the orthogonal polynomials $(q_n)_n$ a differential operator of order less than the one constructed from the function $S$ in Theorem \ref{mainth3}. That orthogonal polynomials $(q_n)_n$ with respect to the Gegenbauer type measure
\eqref{wesp} are eigenfunctions of a differential operator of order $2\alpha+2$ was first proved by
R. Koekoek in 1994 \cite{koe} (the case $\alpha=1$ was discovered by H. Krall in 1940 \cite{Kr2}).

\medskip

\noindent
3.2. The following example is new, as far as the authors know. Consider $m_1=m_2=2$ and $\alpha=\beta\in\mathbb{N}, \alpha\geq 2$. Then we have $2\times2$ matrices $\bm M$ and $\bm N$. Consider for simplicity the case when
$$
\bm M=\begin{pmatrix} \bm M_{0} & \bm M_{1}\\0& 0\end{pmatrix},\quad \bm N=\begin{pmatrix} \bm M_{0} & -\bm M_{1}\\0& 0\end{pmatrix},\quad \bm M_1\neq0,\bm M_0\neq \bm M_1.
$$
The polynomials $(q_n)_n$ are then (left) orthogonal with respect to the inner product (see \eqref{idslip})
\begin{align*}
\langle p,q\rangle&=\int_{-1}^1p(x)q(x)(1-x^2)^{\alpha-2}dx\\
&\hspace{2cm}+\left[p(-1)q(-1)+p(1)q(1)\right]\bm M_0+\left[p(-1)q'(-1)-p(1)q'(1)\right]\bm M_1.
\end{align*}
Again, our assumptions imply \eqref{xesp} and
\begin{align}
\label{zesp1} z_1(x)&=z_3(x)=\frac{4^\alpha(\alpha-2)!}{2}(x+1)(x+2\alpha)+\frac{4(\bm M_0-\bm M_1)}{\alpha!}(x+1)_{2\alpha},\\
\label{zesp2} z_2(x)&=z_4(x)=4^\alpha(\alpha-1)!+\frac{4\bm M_1}{\alpha!}(x+1)_{2\alpha},
\end{align}
where $z_1, z_2$ and $z_3, z_4$  are the polynomials  defined by (\ref{zz1b}) and (\ref{zz2b}), respectively. Hence we have for the polynomials $Y_i, i=1,2,3,4,$ satisfying $Y_i(\theta_x)=z_i(x), i=1,2,3,4$, that $Y_i$ has degree exactly $\alpha$ for all $i=1,2,3,4$.

The associated differential operator $D_S$ in Theorem \ref{mainth3} has order $2(\bwr (\bm M)+\awr (\bm N)+1)=4\alpha +2$. However, again, there is a better choice for the function $S$, in the sense that one can construct from this new $S$ a differential operator of order $2\alpha +2$ for which the orthogonal polynomials $(q_n)_n$ are eigenfunctions. Indeed, consider the rational function
\begin{equation*}\label{defnS2}
S(x)=\frac{\sigma_{x-3/2}R(x)}{\Omega(x)},
\end{equation*}
where $R$ is the polynomial defined by
\begin{align*}
R(x)=16^{\alpha-1}(\alpha-1)!(\alpha-2)!&+2\cdot 4^{\alpha-1}\bm M_0(x-1)_{\alpha-1}(x+\alpha-1)_{\alpha-1}\\
&-4^{\alpha-1}\frac{\bm M_1}{\alpha}(x-2)_\alpha (x+\alpha-1)_\alpha.
\end{align*}
The three assumptions (\ref{ass0}), (\ref{ass1}) and (\ref{ass2}) in Theorem \ref{Teor1} can be easily checked using that
\begin{align*}
M_1(x)&=M_3(x)=\sigma_{x+1}\left(-\frac{\bm M_1}{(\alpha-1)!}(x+3)_{2\alpha-4}\left[\theta_x+(\alpha+1)(2\alpha-1)\right]\right),\\
M_2(x)&=M_4(x)=\sigma_{x+1}\left(2\cdot4^{\alpha-2}(\alpha-2)!+\frac{\bm M_0-\bm M_1}{(\alpha-1)!}(x+3)_{2\alpha-4}\left[\theta_x+(\alpha+1)(2\alpha-1)\right]\right),\\
\tilde M_1(x)&=\tilde M_3(x)=-\frac{\bm M_1}{(\alpha-1)!}(x+3)_{2\alpha-4}\left[\theta_x+(\alpha+1)(2\alpha-1)\right],\\
\tilde M_2(x)&=\tilde M_4(x)=2\cdot4^{\alpha-2}(\alpha-2)!+\frac{\bm M_0-\bm M_1}{(\alpha-1)!}(x+3)_{2\alpha-4}\left[\theta_x+(\alpha+1)(2\alpha-1)\right],\\
\lambda_x&=16^{\alpha-1}(\alpha-1)!(\alpha-2)!x(x+2\alpha-3)+2\frac{4^{\alpha-1}\bm M_0}{\alpha}(x-1)_{2\alpha}-\frac{4^{\alpha-1}\bm M_1}{\alpha(\alpha+1)}(x-2)_{2\alpha+2}.
\end{align*}
For the polynomial $Q(x)=2\lambda_x+\sum_{h=1}^4z_h(x)M_h(x)$, we have
\begin{align*}
Q(x)=&2\cdot16^{\alpha-1}(\alpha-1)!(\alpha-2)!(\theta_x+2(\alpha+1))\\
&+\left[\frac{4^\alpha}{\alpha}(x+3)_{2\alpha-4}\left(\theta_x^2+2(4\alpha^2+\alpha-1)\theta_x+\frac{1}{2}(2\alpha-1)_4\right)\right]\bm M_0\\
&-\left[2\frac{4^{\alpha-1}}{\alpha(\alpha+1)}(x+2)_{2\alpha-2}\left(\theta_x^2+2(4\alpha^2+7\alpha+2)\theta_x+\frac{1}{2}(2\alpha)_4\right)\right]\bm M_1.
\end{align*}
It is now easy to check that $I^{\alpha+\beta}(Q)=Q$ and hence $Q$ is actually a polynomial in $\theta _x$. That is, there exists a polynomial $P_S$ such that
$P_S(\theta_x)=Q(x)$, and hence (\ref{ass2}) holds. Moreover, since $Q$ has degree $2\alpha +2$, the polynomial $P_S$ has degree just $\alpha+1$.

Theorem 3.1 gives that the orthogonal polynomials $(q_n)_n$ are eigenfunctions of the differential operator given by $D_{q,S}$ (\ref{Dq}). From the definition of $\D$-operators for the Jacobi polynomials (see \eqref{dopjac}) it is easy to see that
$$
\sum_{h=1}^4\tilde M_h(D_p)\D_hY_h(D_p)=2\tilde M_1(D_p)\frac{d}{dx}Y_1(D_p)+2\tilde M_2(D_p)\frac{d}{dx}Y_2(D_p).
$$
Now, using the definition of $\tilde M_i(x),Y_i(x) ,i=1,2,$ it is straightforward to see that the degree of $\tilde M_1(x)Y_1(x)+\tilde M_2(x)Y_2(x)$ is at most $2\alpha$. Therefore, the order of the differential operator above is at most $2\alpha+1$. That means that, since $P_S$ is a polynomial of degree $\alpha+1$, the differential operator $D_{q,S}$ (\ref{Dq}) has order equal to $2\alpha +2$.

\medskip

Let us make some comments about the general case of the matrices
$$
\bm M=\begin{pmatrix} M_{00} & M_{01}\\M_{10} & M_{11}\end{pmatrix},\quad \bm N=\begin{pmatrix} N_{00} & N_{01}\\N_{10} & N_{11}\end{pmatrix}.
$$
We have been able to find a differential operator of lower order in the following situation:
$$
\alpha=\beta,\quad z_1(x)=z_3(x)\quad\mbox{and} \quad z_2(x)=z_4(x).
$$
In that case the polynomials $z_1(x)$ and $z_2(x)$ are given by
\begin{align*}
 z_1(x)&=\frac{4^\alpha(\alpha-2)!}{2}(x+1)(x+2\alpha)+\frac{4(M_{00}-M_{01})}{\alpha!}(x+1)_{2\alpha}-\frac{2(M_{10}-M_{11})}{(\alpha+1)!}(x)_{2\alpha+2},\\
z_2(x)&=4^\alpha(\alpha-1)!+\frac{4M_{01}}{\alpha!}(x+1)_{2\alpha}-\frac{2M_{11}}{(\alpha+1)!}(x)_{2\alpha+2},
\end{align*}
while the matrices $\bm M$ and $\bm N$ become
$$
\bm M=\begin{pmatrix} M_{00} & M_{01}\\M_{10} & M_{11}\end{pmatrix},\quad \bm N=\begin{pmatrix} M_{00} & -M_{01}\\-M_{10} & M_{11}\end{pmatrix}.
$$
If we call again $Y_i, i=1,2,$ the polynomials satisfying $Y_i(\theta_x)=z_i(x), i=1,2,$ we see that the degree of $Y_i, i=1,2,$ depends now on the parameters of the matrices $\bm M$ and $\bm N$. In fact we have
$$
\deg Y_1(x)=
\begin{cases}
\alpha+1 & \mbox{if} \quad M_{10}\neq M_{11},\\
\alpha & \mbox{if} \quad M_{10}=M_{11}\quad\mbox{and}\quad M_{00}\neq M_{01},\\
1 &  \mbox{if} \quad M_{10}=M_{11}\quad\mbox{and}\quad M_{00}=M_{01},
\end{cases}
$$
and
$$
\deg Y_2(x)=
\begin{cases}
\alpha+1 & \mbox{if} \quad M_{11}\neq 0,\\
\alpha & \mbox{if} \quad M_{11}=0\quad\mbox{and}\quad M_{01}\neq 0,\\
0 &  \mbox{if} \quad M_{11}=0\quad\mbox{and}\quad M_{01}=0.
\end{cases}
$$
There several combinations, but basically we can summarize them in 3 nontrivial cases.

\medskip
\noindent
\textit{Case 1}. There are three possible situations:
\begin{enumerate}
\item If $\deg Y_1(x)=1$ and $\deg Y_2(x)=\alpha$. Then $M_{11}=M_{10}=0$ and $M_{00}=M_{01}$.
\item If $\deg Y_1(x)=\alpha=\deg Y_2(x)$. Then $M_{11}=M_{10}=0$, $M_{01}\neq0$ and $M_{00}\neq M_{01}$ (this case is the one that we studied in detail above).
\item If $\deg Y_1(x)=\alpha+1$ and $\deg Y_2(x)=0$. Then $M_{11}=M_{01}=0$ and $M_{10}\neq0$.
\end{enumerate}
In any of the three situations above, we have that $2\deg P_S(x)=4\alpha+2$ but we can construct a differential operator of order $2\alpha+2$ for which the corresponding Jacobi-Sobolev orthogonal polynomials are eigenfunctions.

\medskip
\noindent
\textit{Case 2}. There are two possible situations:
\begin{enumerate}
\item If $\deg Y_1(x)=1$ and $\deg Y_2(x)=\alpha+1$. Then $M_{11}\neq0$, $M_{00}=M_{01}$ and $M_{10}=M_{11}$.
\item If $\deg Y_1(x)=\alpha+1=\deg Y_2(x)$ and $\det(\bm M)=0$. Then $M_{11}\neq0$, $M_{10}\neq M_{11}$ and $M_{00}M_{11}=M_{01}M_{10}$.
\end{enumerate}
In any of the two situations above, we have that
$2\deg P_S(x)=4\alpha+6$ but we can construct a differential operator of order $2\alpha+4$ for which the corresponding Jacobi-Sobolev orthogonal polynomials are eigenfunctions.

\medskip
\noindent
\textit{Case 3}. There are three possible situations:
\begin{enumerate}
\item If $\deg Y_1(x)=\alpha$ and $\deg Y_2(x)=\alpha+1$. Then $M_{11}\neq0$, $M_{10}=M_{11}$ and $M_{00}\neq M_{01}$.
\item If $\deg Y_1(x)=\alpha+1$ and $\deg Y_2(x)=\alpha$. Then $M_{11}=0$ and $M_{01},M_{10}\neq0$.
\item If $\deg Y_1(x)=\alpha+1=\deg Y_2(x)$ and $\det(\bm M)\neq0$. Then $M_{11}\neq0$, $M_{10}\neq M_{11}$ and $M_{00}M_{11}\neq M_{01}M_{10}$.
\end{enumerate}
In any of the three situations above, we have that
$2\deg P_S(x)=4\alpha+10$ but we can construct a differential operator of order $2\alpha+6$ for which the corresponding Jacobi-Sobolev orthogonal polynomials are eigenfunctions.

\medskip
For higher dimensions computational evidences indicate that we can find the same phenomenon of lowering the order of the differential operator when
$$
m_1=m_2,\quad \alpha=\beta,\quad\mbox{and}\quad z_l(x)=z_{m_1+l}(x),\quad l=1,\ldots,m_1.
$$
The matrices $\bm M$ and $\bm N$ are then related in the following way:
$$
\bm M=(M_{ij})_{i,j=0}^{m_1-1},\quad \bm N=((-1)^{i+j}M_{ij})_{i,j=0}^{m_1-1}.
$$
The situation gets more complicated and with many more possibilities. The complete study of the general situation is out of the scope of this paper and it will be pursued elsewhere.

\appendix
\section*{Appendix}
\renewcommand{\thesection}{A}

In this appendix we will give the proof of Lemma \ref{l5.2} and the proof of the last part of Theorem \ref{mainth3}.

\subsection*{Proof of Lemma \ref{l5.2}}

In this subsection we will use the following notation. Given a finite set of positive integers $F=\{f_1,\ldots, f_m\}$ (hence $f_i\not=f_j$, $i\not =j$), the expression
\begin{equation}\label{defdosf}
  \begin{array}{@{}c@{}lccc@{}c@{}}
    & &&\hspace{-1.3cm}{}_{j=1,\ldots , m} \\
    \dosfilas{ z_{f,j} }{f\in F}
  \end{array}
\end{equation}
inside a matrix or a determinant denotes the submatrix defined by
$$
\begin{pmatrix}
z_{f_1,1}&z_{f_1,2}&\cdots&z_{f_1,m}\\
\vdots&\vdots&\ddots&\vdots\\
z_{f_m,1}&z_{f_m,2}&\cdots&z_{f_m,m}
\end{pmatrix}.
$$
Given $m$ numbers $u_i$, $i=1,\ldots, m$, and two nonnegative integers $m_1$ and $m_2$ with $m_1+m_2=m$,
we form the pair $\U=(U_1,U_2)$, where $U_1$ is the $m_1$-tuple
$U_1=(u_1,\ldots, u_{m_1})$ and  $U_2$ is the $m_2$-tuple
$U_2=(u_{m_1+1},\ldots, u_{m})$. We also write $\UU_1$ and $\UU_2$ for the sets
\begin{equation}\label{ppmm}
\UU_1=\{1,\ldots, m_1\},\quad \UU_2=\{m_1+1,\ldots, m\}.
\end{equation}

The proof of Lemma \ref{l5.2} is based in the following technical Lemma.

\begin{lemma}\label{lgp1} Let $Y_1, Y_2, \ldots, Y_m,$ be nonzero polynomials satisfying $\deg Y_i=u_{i}$, $i=1,\ldots, m$. Write $r_i$ for the leading coefficient of $Y_i, 1\leq i\leq m$. For real numbers $\alpha,\beta$, consider the rational function $P$ defined by
\begin{equation}\label{def2p}
P(x)=\frac{\left|
  \begin{array}{@{}c@{}lccc@{}c@{}}
    & &&\hspace{-1.5cm}{}_{j=1,\ldots,m} \\
    \dosfilas{N_{1;x-j}^{\alpha;m-j}N_{2;x-1}^{\beta;j-1}Y_{i}(\theta _{x-j}) }{i\in \UU_1} \\
    \dosfilas{Y_{i}(\theta _{x-j}) }{i\in \UU_2}
  \end{array}
   \hspace{-.5cm}\right|}{p(x)q(x)},
\end{equation}
where $p$ and $q$ are the polynomials (\ref{def1p}) and (\ref{def1q}), respectively. The determinant (\ref{def2p}) should be understood in the manner explained above (see (\ref{defdosf})). Then $P$ is a polynomial of degree at most
\begin{equation}\label{ddp}
d=2\sum_{u\in U_1,U_2}u-2\sum_{i=1}^2\binom{m_i}{2}.
\end{equation}
Moreover, if the elements in $U_1$ and $U_2$ are different (i.e. $u_i\not =u_j$, for $i\not =j$, $i,j\in\{1,\ldots ,m_1\}$,
and $u_i\not =u_j$, for $i\not =j$, $i,j\in\{m_1+1,\ldots ,m\}$), then $P$ is a polynomial of degree exactly
(\ref{ddp}) with leading coefficient given by
$$
r=V_{U_1}V_{U_2}\prod_{i=1}^mr_i,
$$
where $V_X$ denotes the Vandermonde determinant associated to the set $X=\{x_1,\ldots,x_K\}$ defined by $V_X=\displaystyle\prod_{i<j}(x_j-x_i)$.
\end{lemma}

\begin{proof}
The Lemma can be proved using the same approach as in the proof of Lemma 3.3 in \cite{dudh}.
\end{proof}

\begin{proof}[Proof of Lemma \ref{l5.2}]

Consider the sets $\UU_j$, $j=1,2$, given by (\ref{ppmm}). By construction (see (\ref{cons})),
we have that for $h\in \UU_j$, the $\mathcal{D}$-operator $\D_h$ is
defined by the sequence $(\varepsilon_{n,j})_n$ (see (\ref{eps1}) and (\ref{eps2})).

Since the polynomial $\Xi$ is invariant under the action of $\I^{\alpha+\beta-m-1}$, we have
\begin{equation}\label{Qprop}
\I^{\alpha+\beta+i}\big(\Xi(x-j)\big)=\Xi(x+m+i+j+1).
\end{equation}
As a consequence of \eqref{Iprop1} and \eqref{ABpropsH} we have
\begin{equation}\label{qprop}
\I^{\alpha+\beta+i}(q(x-j))=(-1)^{\binom{m}{2}}q(x+i+j+m+1),
\end{equation}
and
\begin{equation}\label{pprop}
\I^{\alpha+\beta+i}(p(x-j))=p(x+i+j+m+1),
\end{equation}
where the polynomials $p$ and $q$ are defined by (\ref{def1p}) and (\ref{def1q}), respectively.
\medskip

Now, we check the first assumption (\ref{ass0}) in Theorem \ref{Teor1}, i.e.:
$S(x)\Omega (x)$ is a polynomial in $x$.
From the definition of $S(x)$ in \eqref{GGH} and $\Omega(x)$ in \eqref{casd1} it is straightforward to see, using \eqref{xit}, that
\begin{align}
\label{SOm}& S(x)\Omega(x)=\frac{\sigma_{x-\frac{m-1}{2}}\Xi(x)}{p(x)q(x)}\left|
  \begin{array}{@{}c@{}lccc@{}c@{}}
    & &&\hspace{-1.5cm}{}_{j=1,\ldots,m} \\
    \dosfilas{N_{1;x-j}^{\alpha;m-j}N_{2;x-1}^{\beta;j-1}Y_{i}(\theta _{x-j}) }{i\in \UU_1} \\
    \dosfilas{Y_{i}(\theta _{x-j}) }{i\in \UU_2}
  \end{array}
   \hspace{-.5cm}\right|.
\end{align}
Therefore,
\begin{equation}\label{psw}
S(x)\Omega(x)=\sigma_{x-\frac{m-1}{2}}\Xi (x)P(x),
\end{equation}
where $P$ is the rational function (\ref{def2p}) defined in Lemma \ref{lgp1}. According to this lemma, $P$ is actually a polynomial and thus $S(x)\Omega(x)$ is also a polynomial.

Now we check the second assumption (\ref{ass1}) in Theorem \ref{Teor1}, i.e.: polynomials  $\tilde{M}_1,\ldots,\tilde{M}_m,$ exist such that
\begin{equation*}
M_h(x)=\sigma_{x+1}\tilde{M}_h(\theta_x),\quad h=1,\ldots,m.
\end{equation*}
A simple computation using (\ref{xit}) and Lemma 3.4 in \cite{
ddI2} shows that $M_h$ is actually a polynomial in $x$. It is now sufficient to see that
\begin{equation*}\label{InvMh}
\I^{\alpha+\beta}(M_h(x))=-M_h(x),\quad h=1,\ldots,m,
\end{equation*}
where $M_h(x), h=1,\ldots,m,$ are defined in \eqref{emeiexp}. Hence, $M_h(x), h=1,\ldots,m,$ according to the discussion after \eqref{inv}, is divisible by $\sigma_{x+1}$ and the quotient belongs to $\mathbb{R}[\theta_x]$.

We assume that the $h$-th $\mathcal{D}$-operator is $\mathcal{D}_1$ (similar proof for $\mathcal{D}_2$). As before, we can remove all the denominators in $M_h(x)$ in this case and rearrange the determinant to obtain
\begin{equation*}
M_h(x)=\sum_{j=1}^m(-1)^{h+j}\frac{\sigma_{x+j-\frac{m-1}{2}}\Xi(x+j)}{p(x+j)q(x+j)}N_{1;x}^{\alpha;m-j}N_{2;x+j-1}^{\beta;j-1}\left|
  \begin{array}{@{}c@{}lccc@{}c@{}}
        &&&\hspace{-1.5cm}{}_{l\neq h, r\neq j} \\
        \dosfilas{N_{1;x+j-r}^{\alpha;m-r}N_{2;x+j-1}^{\beta;r-1}Y_l(\theta_{x+j-r})}{l\in \UU_1\setminus\{h\}}\\
    \dosfilas{Y_l(\theta_{x+j-r})}{l\in \UU_2}
\end{array}
   \hspace{-0.5cm} \right|.
\end{equation*}
Hence, using \eqref{ABpropsH}, \eqref{Qprop}, \eqref{qprop} and \eqref{pprop}, we have
\begin{align*}
&\I^{\alpha+\beta}\big(M_h(x)\big)=-\sum_{j=1}^m(-1)^{h+j}(-1)^{\binom{m}{2}}\frac{\sigma_{x+\frac{m-1}{2}-j+2}\Xi(x+m-j+1)}{p(x+m-j+1)q(x+m-j+1)}\times\\
&\hspace{5cm}\times N_{2;x+m-j}^{\beta;m-j}N_{1;x}^{\alpha;j-1}\left|
  \begin{array}{@{}c@{}lccc@{}c@{}}
        &&&\hspace{-1.5cm}{}_{l\neq h, r\neq j} \\
        \dosfilas{N_{2;x+m-j}^{\beta;m-r}N_{1;x+r-j}^{\alpha;r-1}Y_l(\theta_{x-j+r})}{l\in \UU_1\setminus\{h\}}\\
    \dosfilas{Y_l(\theta_{x-j+r})}{l\in \UU_2}
\end{array}
   \hspace{-0.5cm} \right|\\
 &=-(-1)^{\binom{m}{2}}(-1)^{m-1}(-1)^{\binom{m-1}{2}}\sum_{j=1}^m(-1)^{h+j}\frac{\sigma_{x+j-\frac{m-1}{2}}\Xi(x+j)}{p(x+j)q(x+j)} \times\\
&\hspace{5cm}\times N_{1;x}^{\alpha;m-j}N_{2;x+j-1}^{\beta;j-1}\left|
  \begin{array}{@{}c@{}lccc@{}c@{}}
        &&&\hspace{-1.5cm}{}_{l\neq h, r\neq j} \\
        \dosfilas{N_{1;x+j-r}^{\alpha;m-r}N_{2;x+j-1}^{\beta;r-1}Y_l(\theta_{x+j-r})}{l\in \UU_1\setminus\{h\}}\\
    \dosfilas{Y_l(\theta_{x+j-r})}{l\in \UU_2}
\end{array}
   \hspace{-0.5cm} \right|\\
&=-(-1)^{m(m-1)}M_h(x)=-M_h(x).
\end{align*}
Note that we renamed the index $j$ ($j\to m-j+1$) in the second step and that we interchanged all columns in the determinant ($r\to m-r+1$), thus the corresponding change of signs.

Finally, we check the third assumption (\ref{ass2}) in Theorem \ref{Teor1}, i.e.: a polynomial $P_S$ exists such that
\begin{equation*}
P_S(\theta_x)=2\lambda_x+\sum_{h=1}^mY_h(\theta_x)M_h(x).
\end{equation*}
As it was pointed out in \cite{ddI2} (see (5.8)) it is sufficient to see that
$$
\I^{\alpha+\beta-1}\big(S(x)\Omega(x)\big)=-\big(S(x+m)\Omega(x+m)\big).
$$
From \eqref{SOm}, using \eqref{ABpropsH}, \eqref{Qprop}, \eqref{qprop} and \eqref{pprop} again, we have

\begin{align*}
& \I^{\alpha+\beta-1}\big(S(x)\Omega(x)\big)=-(-1)^{\binom{m}{2}}\frac{\sigma_{x+\frac{m+1}{2}}\Xi(x+m)}{p(x+m)q(x+m)}\left|
  \begin{array}{@{}c@{}lccc@{}c@{}}
    & &&\hspace{-1.5cm}{}_{j=1,\ldots,m} \\
    \dosfilas{N_{2;x+m-1}^{\beta;m-j}N_{1;x+j-1}^{\alpha;j-1}Y_{i}(\theta _{x+j-1}) }{i\in \UU_1} \\
    \dosfilas{Y_{i}(\theta _{x+j-1}) }{i\in \UU_2}
  \end{array}
   \hspace{-.5cm}\right|\\
=&-(-1)^{\binom{m}{2}}(-1)^{\binom{m}{2}}\frac{\sigma_{x+m-\frac{m-1}{2}}\Xi(x+m)}{p(x+m)q(x+m)}\left|
  \begin{array}{@{}c@{}lccc@{}c@{}}
    & &&\hspace{-1.5cm}{}_{j=1,\ldots,m} \\
    \dosfilas{N_{2;x+m-1}^{\beta;j-1}N_{1;x+m-j}^{\alpha;m-j}Y_{i}(\theta _{x+m-j}) }{i\in \UU_1} \\
    \dosfilas{Y_{i}(\theta _{x+m-j}) }{i\in \UU_2}
  \end{array}
   \hspace{-.5cm}\right|\\
=&-S(x+m)\Omega(x+m).
\end{align*}
\end{proof}

\subsection*{Proof of the last part of Theorem \ref{mainth3}}
It remained to prove the computation of the order of  the operator $D_{S}$ in Theorem \ref{mainth3}. For that, we will give three auxiliary lemmas.  We need first to introduce some notation. Given $m$ arbitrary polynomials $Y_1,\ldots , Y_m$, we will denote by $\cY $ the $m$-tuple of polynomials $(Y_1,\ldots, Y_m)$. The $m$-tuple formed by interchanging the polynomials $Y_i$ and $Y_j$ in $\cY$ is denoted by $\cY_{i\leftrightarrow j}$; the $m$-tuple formed by changing the polynomial $Y_i$ to $aY_i+bY_j$ in $\cY$, where $a$ and $b$ are real numbers, is denoted by $\cY_{i\leftrightarrow ai+bj}$; and the $m$-tuple formed by removing the polynomial $Y_i$ in $\cY$ is denoted by $\cY_{\{ i\} }$.

\begin{lemma}\label{l6.1}
Given $m$ arbitrary polynomials $Y_1,\ldots , Y_m$, we form the $m$-tuple of polynomials
$\cY =(Y_1,\ldots, Y_m)$ and consider the  operator $D_{q,S}=D_{q,S}(\cY )$ (\ref{Dq}). Then, for any numbers $a, b\in\mathbb{R}$ we have
\begin{align} \label{Dqinv1}
D_{q,S}(\cY )&=-D_{q,S}(\cY_{i\leftrightarrow j}),\\\label{Dqinv2}
D_{q,S}(\cY_{i\leftrightarrow ai+bj})&=aD_{q,S}(\cY).
\end{align}
\end{lemma}

\begin{proof}
It is analogous to the proof of Lemma 3.4 in \cite{ddI1}.
\end{proof}

\begin{lemma}\label{l6.2} Using the same notation as the one employed in Theorem \ref{Teor1}, we write
$$
\Psi_j^h(x)=\xi_{x-j,m-j}^hS(x)
\det\left(\xi_{x-r,m-r}^{l}Y_l(\theta_{x-r})\right)_{l\in \II_h;r\in \II_j}, \quad h,j=1,\ldots , m,
$$
and $\Omega _g^h$, $h=1,\ldots , m$, $g=0,1,2,\ldots$, for the particular case of $\Omega$ when $Y_h(x)=x^g$.
We assume that $\theta_x$ is a polynomial in $x$ of degree $2$, and that $\Psi_j^h$, $h,j=1,\ldots , m$, are polynomials in $x$, and we write $\tilde d=\max\{\deg \Psi_j^h:h,j=1,\ldots , m\}$.
Then, $M_h$ and $S\Omega _g^h$, $h=1,\ldots , m$, $g=0,1,2,\ldots$, are also polynomials in $x$. In addition, we assume that for each $h$ there exists $g_h$ such that $S\Omega_{g}^h=0$, $g=0,\ldots, g_h$, and that the degree of $S\Omega _g^h$ is at most $2(g-g_h)+\deg (S\Omega_{g_h}^h)$
for $g\le \tilde d-\deg (S\Omega_0^h)$. Then $M_h$ is a polynomial of degree at most $\deg (S\Omega_{g_h}^h)-2g_h$.
\end{lemma}

\begin{proof}
This lemma is a variant of Lemma 3.4 in \cite{ddI2}, and can be proved as
Lemma 3.2 in \cite{dudh}.
\end{proof}

\begin{lemma}\label{l6.3}
For $m_1,m_2\geq0$ with $m=m_1+m_2\geq1$, let $\bm M=(M_{i,j})_{i,j=0}^{m_1-1}$ and $\bm N=(N_{i,j})_{i,j=0}^{m_2-1}$ be $m_1\times m_1$ and $m_2\times m_2$ matrices, respectively. If $\bm M ,\bm N \not =0$, we assume, in addition, that $\alpha$ and $\beta$ are nonnegative integers with $\alpha \ge m_2$ and $\beta \ge m_1$. If, instead, $\bm M=0$,  we assume that only $\alpha$ is a positive integer with $\alpha \ge m_2$, and if $\bm N=0$, we assume that only $\beta$ is a positive integer with $\beta \ge m_1$. For $j=1,\ldots , m$, define the polynomials $Y_l$, $Y_l(\theta_x)=z_l(x)$, where $z_l$ is defined by (\ref{zz1b}) and (\ref{zz2b}). Then the degree of the polynomial $P$ defined by (\ref{def2p}) is $2(\bwr (\bm M)+\awr (\bm N))$.
\end{lemma}

\begin{proof}
It is analogous to the proof Lemma 4.1 in \cite{ddI1}.
\end{proof}

\begin{proof}[Proof of the last part of Theorem \ref{mainth3}]

We will compute the order of the operator $D_{S}$ in Theorem \ref{mainth3}. Given $m$ numbers $u_i$, $i=1,\ldots, m$, and two nonnegative integers $m_1$ and $m_2$ with $m_1+m_2=m$,
we form the pair $\U=(U_1,U_2)$, where $U_1$ is the $m_1$-tuple
$U_1=(u_1,\ldots, u_{m_1})$ and  $U_2$ is the $m_2$-tuple
$U_2=(u_{m_1+1},\ldots, u_{m})$.
Given a polynomial $\Xi$ which is invariant under the action of $\I^{\alpha+\beta-m-1}$, we associate the rational function $S$ as in (\ref{GGH}). Given $m$ polynomials $Y_i$, $i=1,\ldots, m$, with $\deg Y_i=u_i$, we consider the (quasi) Casorati determinant $\Omega$ and the polynomial $P$ as in (\ref{def2p}). As established in the proof of Lemma \ref{l5.2} (see (\ref{psw})), we have $S(x)\Omega(x)=\sigma_{x-\frac{m-1}{2}}\Xi(x)P(x)$, and then
\begin{equation}\label{grs}
\mbox{the degree of $S\Omega $ is $\deg \Xi +d+1$, where $d$ is the degree of $P$.}
\end{equation}

Consider now the polynomials $Y_l$, defined by $Y_l(\theta_x)=z_l(x)$, $l=1,\ldots , m$, where  $z_l$ are the polynomials (\ref{zz1b}) and (\ref{zz2b}).
The operator $D_{q,S}$ (\ref{Dq}) is the sum of the operators $T_1=\frac{1}{2}P_S(D_{p})$ and $T_2=\sum_{h=1}^m \tilde M_h(D_{p})\D_h Y_h(D_{p})$. Since the order of the differential operator $D_{p}$ is $2$, it is clear from the definition of the polynomial $P_S$ that the order of $P_S(D_{p})$ is just $2\deg P_S$. According to (\ref{Pdiff2}), we have that $2\deg P_S=\deg (S\Omega)+1$.
Using (\ref{grs}), we get $\deg P_S=\deg(\Xi)/2+d/2+1$. Hence
\begin{equation}\label{dddd}
\mbox{the order of $T_1$ is $\deg(\Xi)+d+2$,}
\end{equation}
where $d$ is the degree of the polynomial
$P$ associated to the polynomials $Y_l$ (see (\ref{def2p})).
Using Lemma \ref{l6.3}, we then get that
the order of $T_1$ is $\deg(\Xi)+2(\bwr (\bm M)+\awr (\bm N))+2$.
It is now enough to prove that the order of the operator $T_2$ is less than the order of $T_1$.

To stress the dependence of $P_S$, $P$, $M_h$, $\tilde M_h$, $\Omega$ and the operator $D_{q,S}$ on the $m$-tuple of polynomials
$\cY =(Y_1,\ldots ,Y_m)$, we write  $P_S=P_S(\cY )$, $P=P(\cY )$, $M_h=M_h(\cY )$, $\tilde M_h=\tilde M_h(\cY )$, $\Omega=\Omega(\cY)$ and $D_{q,S}=D_{q,S}(\cY )$.

Interchanging and using linear combinations of two polynomials, we can get from the polynomials $Y_i$, $i=1,\ldots , m$, new polynomials $\hat Y_i$, $i=1,\ldots , m$, satisfying that
\begin{align}\nonumber
& \mbox{$\deg \hat Y_i\not =\deg \hat Y_{j}$, $i\not =j$, $1\le i,j \le m_1$, or $m_1+1\le i,j \le m$.}\\\label{mmm}
& \mbox{$\deg \hat Y_i$ is increasing from $i=1,\ldots, m_1$, and from $i=m_1+1,\ldots, m$.}\\\label{mm}
& \mbox{Fixing $h$, $1\le h\le m_1$, define $g_h$ as the first nonnegative integer such that}\\\nonumber
& \mbox{$g_h\not \in \{\deg \hat Y_i,i=1,\ldots , m_1\}$. Then for $0\le g< g_h$, $\hat Y_g=x^g$. The same for}\\\nonumber
& \mbox{$m_1+1\le h\le m$.}
\end{align}
Write $\hat U_1=\{\deg \hat Y_i, 1\le i\le m_1\}$ and $\hat U_2=\{\deg \hat Y_i, m_1+1\le i\le m\}$.

Using the invariance properties (\ref{Dqinv1}) and (\ref{Dqinv2}),
we then have
\begin{align*}
P_S&=P_S(\cY )=P_S(\hat \cY ),\\
D_{q,S}&=D_{q,S}(\cY )=D_{q,S}(\hat \cY ),
\end{align*}
where $\hat \cY =(\hat Y_1,\ldots ,\hat Y_m)$.
If we write $\hat M_h= M_h(\hat \cY )$, $\hat {\tilde M}_h=\tilde M_h(\hat \cY )$, $h=1,\ldots , m$, we then have
\begin{equation}\label{ot2}
T_2=\sum_{h=1}^m\hat{\tilde M}_h(D_{p})\D \hat Y_h(D_{p}).
\end{equation}
Since $T_1=P_S(\cY )=P_S(\hat \cY )$, we have as before
that the order of $T_1$ is $\deg(\Xi)+d+2$, where $d$ is the degree of the polynomial $P(\hat \cY)=P(\cY)$, which according to Lemma \ref{lgp1} is
((\ref{mmm}) says that the elements in $\hat U_1$ and $\hat U_2$ are different, respectively)
\begin{equation}\label{ddd}
d=2\sum_{u\in \hat U_1,\hat U_2}u-2\sum_{i=1}^2\binom{m_i}{2}.
\end{equation}
A straightforward computation using (\ref{ot2}) shows that the order of the operator $T_2$ is less than or equal to
$$
\max \{ 2\deg \hat{\tilde M}_h+2\deg \hat Y_h +1, h=1,\ldots ,m \}.
$$
From (\ref{ass1}), we get that $\deg \hat{\tilde M}_h=(\deg \hat{M}_h-1)/2$. Hence
the order of the operator $T_2$ is less than or equal to
$$
\max \{ \deg \hat{ M}_h+2\deg \hat Y_h , h=1,\ldots ,m \}.
$$
Fixed now $h$, $1\le h\le m_1$ (the same for $m_1+1\le h\le m$).
We  write $\hat \cY ^h_g$ for the $m$-tuple formed by changing the polynomial $\hat Y_h$ to $x^g$ in $\hat \cY$, and write
$\hat \Omega^h_g=\Omega (\hat \cY^h_g)$. Using (\ref{mm}), we get that $\hat \Omega ^h_g=0$ for $0\le g<g_h$ (there are two equal rows) and hence $S\hat \Omega ^h_g=0$ for $0\le g<g_h$. Using (\ref{grs}), we have that  $\deg (S\hat \Omega^h_g)=\deg \Xi +\deg P(\hat \cY^h_g)+1$.
(\ref{mmm}) and (\ref{mm}) show that the degrees of the $m_1$ first polynomials in $\cY^h_{g_h}$ are different, as well as those of the $m_2$ last polynomials. Then Lemma \ref{lgp1} gives
\begin{equation}\label{m4m}
\deg P(\hat \cY^h_{g_h})=2\sum_{u\in\left(\hat U_1\setminus \{\deg \hat Y_h\},\hat U_2\right)}u+2g_h-2\sum_{i=1}^2\binom{m_i}{2}.
\end{equation}
For $g>g_h$, Lemma \ref{lgp1} also gives
$$
\deg P(\hat \cY^h_{g})\le 2\sum_{u\in\left(\hat U_1\setminus \{\deg \hat Y_h\},\hat U_2\right)}u+2g-2\sum_{i=1}^2\binom{m_i}{2}=
\deg P(\hat \cY^h_{g_h})+2(g-g_h).
$$
Using (\ref{grs}), we get that
 $\deg (S\hat \Omega ^h_g)\le 2(g-g_h)+ \deg (S\hat \Omega ^h_{g_h})$ for $g\ge g_h$.

Using now Lemma \ref{l6.2}, we get that the degree of the polynomial $\hat M_h$ is less than or equal to $\deg (S\hat \Omega ^h_{g_h})-2g_h$. Hence,
using (\ref{grs}), (\ref{ddd}) and (\ref{m4m}), we have
\begin{align*}
\deg \hat{ M}_h+2\deg \hat Y_h & \le \deg (S\hat \Omega ^h_{g_h})-2g_h+2\deg \hat Y_h\\
&=\deg \Xi +\deg P(\hat \cY^h_{g_h})+1-2g_h+2\deg \hat Y_h
\\ & =\deg \Xi+2\sum_{u\in \left(\hat U_1,\hat U_2\right)}u-2\sum_{i=1}^2\binom{m_i}{2}+1\\
& =\deg \Xi +d+1.
\end{align*}
Comparing with (\ref{dddd}), this gives that the order of the operator $T_2$ is less than  the order of $T_1$. This completes the proof of Theorem \ref{mainth3}.
\end{proof}

\end{document}